\documentclass[10pt, makeidx]{amsart} %\usepackage[active]{srcltx}
\usepackage{pstricks}
\usepackage{pspicture}
\usepackage{psfrag}
\usepackage{epsfig}
\usepackage{amsmath}
\usepackage{amssymb}
\usepackage{curves}
\usepackage{latexsym}
\usepackage{amsfonts,amsthm}
\usepackage{fancyhdr}
\usepackage{amsxtra,enumerate}
\usepackage{stmaryrd}
\usepackage{hhline}
\usepackage{epic}
\usepackage{eepic}
\usepackage{bm}
\usepackage{graphics}
\usepackage{url}
\usepackage[top=1in, bottom=1in, left=1in, right=1in]{geometry}
\usepackage{hyperref}
\usepackage{graphicx,color}
\newtheorem{theorem}{Theorem}[section]

\theoremstyle{definition}

\newtheorem{proposition}{Proposition}[section]
\newtheorem{Remark}{Remark}
\numberwithin{equation}{section}
\numberwithin{figure}{section}
\usepackage{scrextend}
\deffootnote[1em]{1em}{1em}{%
  \textsuperscript{\makebox[1em][l]{\thefootnotemark}}}

\def\blackslug{\hbox{\hskip 1pt \vrule width 4pt height 8pt depth 1.5pt
\hskip 1pt}}
\def\qed{\quad\blackslug\lower 8.5pt\null\par}

\makeindex

\begin{document}

\author{Charis Tsikkou} 
\address{Mathematics Department\\West Virginia University \\Morgantown, WV 26506} 
\email{tsikkou@math.wvu.edu}

\title[A Riemann Solution with Singular Shocks]{Singular Shocks in a Chromatography Model}
\maketitle
 
%\bigskip
\begin{center}
\today
\end{center}
\begin{abstract}
We consider a system of two equations that can be used to describe nonlinear chromatography and produce a coherent explanation and description of the unbounded solutions (singular shocks) that appear in Mazzotti's model \cite{Mazzotti2}. We use the methods of Geometric Singular Perturbation Theory, to show existence of a viscous solution to Dafermos-DiPerna regularization. 
\end{abstract}
\maketitle
\setcounter{tocdepth}{1}
\numberwithin{equation}{subsection}
\section{\textbf{Introduction}}\
\\
The aim of this paper is to show existence of no classical Riemann solutions to a physical model with important applications in modern industry. It has been already shown in carefully designed experiments by Mazzotti et al. \cite{Mazzotti1, Mazzotti2, Mazzotti3}, that this model exhibits singular shocks.

Singular shocks, a type of weak solutions of very low regularity have been studied before. There were originally discovered by Keyfitz and Kranzer \cite{B.L.KeyfitzandH.C.Kranzer, wgt}, and later studied in greater depth by Sever \cite{Sever}. Keyfitz and Kranzer \cite{B.L.KeyfitzandH.C.Kranzer} worked with a strictly hyperbolic, genuinely nonlinear system derived from a $1$-dimensional model for isothermal, isentropic gas dynamics and they observed that there is a large region, where the Riemann problem cannot be solved using shocks and rarefactions. They produced approximate unbounded solutions which do not satisfy the equation in the classical weak-solution sense and showed that only the first component of the Rankine-Hugoniot relation is satisfied, giving a unique speed $s$ for which any given two states $U_L$ and $U_R$ can be joined. Later on, Schecter \cite{Schecter} proved existence of a viscous solution following Dafermos's approach \cite{C.M.Dafermos, DafermosDiPerna}, under the condition that the singular shock is overcompressive. Schecter used a geometric method and dynamical systems theory (blowing-up approach to geometric singular perturbation problems that lack normal hyperbolicity, see Fenichel \cite{F} and Jones \cite{J}).  

Keyfitz and Tsikkou \cite{KT}, showed existence of approximations to singular shock solutions by the same method, for a non hyperbolic system (change of type) derived from isentropic gas dynamics for an ideal fluid with $1<\gamma<\frac{5}{3},$ conserving velocity and entropy. Singular shocks also appear in a two-fluid model for incompressible two-phase flow, see Keyfitz et al. \cite{Ke, KeSaSe, KeSeZh}, in a model describing gravity-driven, particle-laden thin-film flow, see Wang and Bertozzi \cite{WB}, Mavromoustaki and Bertozzi \cite{MaBe}, in the Brio system appearing in the study of plasma and the classical shallow-water system see Kalisch and Mitrovic \cite{KaMi} and possibly in a model for chemotaxis, see Levine and Sleeman \cite{LeSl}. 

Naturally, questions then arise about whether it is possible to predict singular shock solutions to systems, find a physical interpretation of their significance, explain the sense in which they satisfy the equation and find a better definition which will describe some wider collection of examples, check for connection between singular shocks, genuinely nonlinear systems and change of type. A few of these questions will be subject of future work.

Investigation of singular solutions was mostly focused on the case when only one state variable develops the Dirac delta function and the others are functions with a bounded variation. We have though other physically important systems with delta functions in more than one state variables. For example, Mazzotti \cite{Mazzotti1, Mazzotti2, Mazzotti3} in his recent work, showed numerically and experimentally that the following model, in a single space dimension and time, arising in two-component chromatography (concentration $u_i$ for chemical i)
\begin{equation}
\frac{\partial}{\partial t}(u_i+\frac{\alpha_i u_i}{1-u_1+u_2})+\frac{\partial u_i}{\partial x}=0, \ i=1, 2, \  \alpha_1<\alpha_2.
\label{chromatography}
\end{equation}
exhibits singular solutions. He obtained approximate solutions, using a linear combination of $\delta$-functions, with an error that converged to zero and showed that neither of the Rankine-Hugoniot equations is satisfied. In system (\ref{chromatography}) that results when some assumptions in the traditional Langmuir equilibrium model are changed, the conserved quantities are the masses of two components flowing at constant speed along a column, cooperating for adsorption sites and is a system which exhibits change of type (hyperbolic and elliptic).

In this paper, we obtain useful information from the Dafermos-DiPerna self-similar regularization and produce an explanation/description of the singular solution in Mazzotti's work. 

In the next section, we derive a simpler system of equations which we will study, by rescaling time and changing the dependent variables. These changes are linear in the conserved quantities so that the form of the system is maintained. Derivation of alternative models will be also subject of future work. In Section 3 we give a formal description of the Riemann solutions, including the cases that include vacuum states. As in Keyfitz et al. \cite{B.L.KeyfitzandH.C.Kranzer, KT}, we are led to the form of the solutions by using a self-similar viscous perturbation of the system. The new system has now similar properties to those in \cite{B.L.KeyfitzandH.C.Kranzer, KT}, as only the first component of the Rankine-Hugoniot relation is satisfied. In Section 4, we use the theory of dynamical systems in the same spirit as in Schecter \cite{Schecter}, Keyfitz and Tsikkou \cite{KT} and more specifically geometric singular perturbation theory (GSPT), see Fenichel \cite{F}, Jones \cite{J}, Krupa and Szmolyan \cite{KS}, Jones and Kopell \cite{JK}, Schecter and Szmolyan \cite{SS}, to construct orbits that connect the left and right states given by \begin{equation}
U(x,0)=\left\{
                      \begin{array}{ll}
                        U_L, & \hbox{$x<0$;} \\
                        U_R, & \hbox{$x\geq 0.$}
                      \end{array}
                    \right.
\label{initialdata}
\end{equation} We also prove existence of self-similar viscous profiles for overcompressive singular shocks for the chromatography model. It should be noted, however, that the symmetry in the orbits is lost and the solutions differ significantly from those of previous models exhibiting singular shocks. 

There is a body of literature on all kinds of chromatography systems of the form
\begin{equation}
\label{kinds}
(u_j)_t+(u_jf_j(\mu_1 u_1+\ldots+\mu_n u_n))_x=0, \ \ x\in\mathbb{R}, \ \ t \geqslant 0, \ \ j=1, 2,
\end{equation} but all the known results are for simplified, everywhere hyperbolic, systems which also belong to Temple class. For recent developments in this direction see Shen \cite{Shen}, Li and Shen \cite{LS} and Sun \cite{Sun1, Sun2} for a system with $f_j(w)=\frac{1}{1+w}, \ \ \mu_1=1, \ \ \mu_2=0$ (the second characteristic family is linearly degenerate); Guo, Pan and Yin \cite{GPY}, Cheng and Yang \cite{CY} for a system with $f_j(w)=1+\frac{1}{1+w}, \ \ \mu_1=-1, \ \ \mu_2=1$ (the first characteristic family is linearly degenerate); Wang \cite{Wang} for a system with $f_j(w)=\frac{1}{1+w}, \ \ \mu_1=-1, \ \ \mu_2=1$ (the first characteristic family is linearly degenerate), and the references cited therein. See also Shelkovich \cite{S2} 
for a class of systems with a different definition of solutions whose components contain Dirac delta functions. These include the system of nonlinear chromatography for $f_j(w)=1+\frac{a_j}{1+w}, \ \ \mu_j=1$ and $a_j$ is Henry's constant. 
 
\section{\textbf{Preliminaries}}\
\\
In this section we derive a simpler system of equations from (\ref{chromatography}) which we study in this paper. Since these changes are linear in the conserved quantities we are not changing the form of the system. We also analyze its basic properties (hyperbolicity, genuine nonlinearity, and the shock and rarefaction curves sketched in Figure \ref{Figure2.1}). Finally, we identify the regions where classical Riemann solutions exist.
\subsection{\textbf{Derivation of the model from chromatography}}\
\\
We start out with the equations in the form used by Mazzotti \cite{Mazzotti2},
\begin{equation}
\left\{
    \begin{array}{ll}
    (u_1+\frac{\alpha_1 u_1}{1-u_1+u_2})_{\tau}+(u_1)_x=0,\\ \\
    (u_2+\frac{\alpha_2 u_2}{1-u_1+u_2})_{\tau}+(u_2)_x=0,
    \end{array}
  \right.
\label{2.1.1}
\end{equation}
with $\alpha_1<\alpha_2.$ To create a system more conventional to conservation laws researchers, we make some changes of variables. First, we change to a moving coordinate system, or rescale time: $$x'=x, \ \  t=\tau-x,$$ so that the system becomes
\begin{equation}
\left\{
    \begin{array}{ll}
    (\frac{\alpha_1 u_1}{1-u_1+u_2})_{t}+(u_1)_{x'}=0,\\ \\
    (\frac{\alpha_2 u_2}{1-u_1+u_2})_{t}+(u_2)_{x'}=0.
    \end{array}
  \right.
\label{2.1.2}
\end{equation}
We then drop the prime in $x.$ The aim is to focus on the conserved quantities $v_1=\frac{\alpha_1 u_1}{1-u_1+u_2}$ and $v_2=\frac{\alpha_2 u_2}{1-u_1+u_2}$ so we also change the dependent variables. If we let $$\omega_1=\frac{u_1}{1-u_1+u_2}, \ \ \omega_2=\frac{u_2}{1-u_1+u_2},$$
then we have
\begin{equation}
\left\{
    \begin{array}{ll}
    (\omega_{1})_t+(\frac{u_1}{\alpha_1})_x=0, \\ \\
    (\omega_{2})_t+(\frac{u_2}{\alpha_2})_x=0.
    \end{array}
  \right.
\label{2.1.3}
\end{equation}
Looking then at (\ref{2.1.3}), we let $$v=(\alpha_1 \alpha_2)^{1/3}(1+\omega_1-\omega_2), \ \ y=\frac{1}{(\alpha_1 \alpha_2)^{1/3}}[\alpha_2 \omega_1-\alpha_1 \omega_2-(\alpha_1+\alpha_2)v],$$
and we find
\begin{equation}
\left\{
    \begin{array}{ll}
    v_{t}+(\frac{y}{v})_x=0, \\ \\
    y_{t}+(\frac{1}{v})_x=0.
    \end{array}
  \right.
\label{2.1.4}
\end{equation}
In the original variables $u_1$ and $u_2$ the new variables can be expressed as
$$\frac{v}{(\alpha_1 \alpha_2)^{1/3}}=\frac{1}{1-u_1+u_2}, \ \ (\alpha_1 \alpha_2)^{1/3}y=\frac{\alpha_2 u_1-\alpha_1 u_2-(\alpha_1+\alpha_2)}{1-u_1+u_2}.$$
This system, equivalent to (\ref{2.1.1}) for smooth solutions, but possessing different weak solutions, expresses conservation of $v$ and $y.$ We define $U=(v,y)^\intercal$ and $F=F(U)=(\frac{y}{v},\frac{1}{v})^\intercal$ the flux function.
We work with the system (\ref{2.1.4}) and Riemann data
\begin{equation}
U(x,0)=\left(
         \begin{array}{c}
           v \\
           y \\
         \end{array}
       \right)
(x,0)=\left\{
                      \begin{array}{ll}
                        U_L, & \hbox{$x<0$;} \\
                        U_R, & \hbox{$x\geq 0$.}
                      \end{array}
                    \right.
\label{2.1.5}
\end{equation}
to show existence of singular shocks. Attention is drawn to the limit $v\rightarrow 0$ where the variables $u_1,$ $u_2$ of (\ref{2.1.1}) become singular. 

\subsection{\textbf{Hyperbolicity and Genuine Nonlinearity}}\
\\
The Jacobian of (\ref{2.1.4}) is
\begin{equation}
\left(
  \begin{array}{cc}
    -\frac{y}{v^2} & \frac{1}{v} \\ \\
    -\frac{1}{v^2} & 0 \\
  \end{array}
\right).
\label{2.2.1}
\end{equation}
The eigenvalues of (\ref{2.2.1}) are
\begin{align}
\lambda_1(v,y)=\frac{-y-\sqrt{y^2-4v}}{2v^2}, \label{2.2.2} \\
\lambda_2(v,y)=\frac{-y+\sqrt{y^2-4v}}{2v^2}. \label{2.2.3}
\end{align}
The eigenvectors are
\begin{align}
r_1&=\left(
      \begin{array}{c}
        2v \\ \\
        y-\sqrt{y^2-4v} \\ 
      \end{array}
    \right), \label{2.2.4} \\ \nonumber \\
r_2&=\left(
      \begin{array}{c}
        2v \\ \\
        y+\sqrt{y^2-4v} \\ 
      \end{array}
    \right)
. \label{2.2.5}
\end{align}
The system (\ref{2.1.4}) is strictly hyperbolic when $4v<y^2,$ and non-hyperbolic when $4v>y^2.$ On $y^2=4v,$ $\lambda_1=\lambda_2,$ and $r_1=r_2.$

For the system (\ref{2.1.1}), since $v_i$ as well as $u_i$ must be positive, Mazzotti, considered only states with $1-u_1+u_2>0$ and data in the hyperbolic part of state space in the closure of the open component neighboring the origin. This physically meaningful experimental situation for (\ref{2.1.4}) corresponds to the region bounded by a curvilinear triangle with vertices
$$O=\left(\alpha, -\frac{\alpha_1+\alpha_2}{\alpha}\right), \ \ A=\left(\frac{\alpha_1}{\alpha_2}\alpha, -\frac{2\alpha_1}{\alpha}\right), \ \ B=\left(\frac{\alpha_2}{\alpha_1}\alpha, -\frac{2\alpha_2}{\alpha}\right),$$
where $\alpha=(\alpha_1 \alpha_2)^{1/3}$ and sides
$$\text{OA:}\ \  y=-\frac{\alpha_2 v}{\alpha^2}-\frac{\alpha_1}{\alpha},$$
$$\text{OB:}\ \  y=-\frac{\alpha_1 v}{\alpha^2}-\frac{\alpha_2}{\alpha}.$$
Therefore $v>0,$ $y<0$ and $\lambda_1(v,y), \ \ \lambda_2(v,y)>0.$
Since $D \lambda_i r_i\neq 0$ if $y^2\neq\frac{16}{3}v$ then the states below $y^2=4v$ and above $y^2=\frac{16}{3}v$ are genuinely nonlinear for both $i$-characteristic families. To stay in the strictly hyperbolic, genuinely nonlinear physically feasible region we need $\frac{\alpha_2}{3}<\alpha_1<3\alpha_2.$
\subsection{\textbf{Rarefaction Curves Through the Left State $U_L$ in the Hyperbolic Region}}\
\\
For $i=1$ or $2,$ the $i$-rarefaction curves are solutions of the system
\begin{equation}
\left(\begin{array}{c}
\dot{v} \\
\dot{y} \\
\end{array}
\right)=\left(
\begin{array}{c}
2v \\
y\mp\sqrt{y^2-4v} \\
\end{array}
\right),
\label{2.3.1}
\end{equation}
where overdot denotes derivative with respect to $\xi=\lambda_i(v,y).$
By the change of variables $w=\sqrt{\frac{y^2-4v}{v}},$ we get $\frac{d}{dv}(w)=\pm \frac{\sqrt{w^2+4}}{2v}$ Upon separation of the variables, integration, further calculations, and returning to the $U$ variables we derive
\begin{align}
R_1(v_L,y_L): \sqrt{y^2-4v}-y=& \frac{v}{v_L}(\sqrt{y_L^2-4v_L}-y_L),
\label{2.3.2} \\
R_2(v_L,y_L): \sqrt{y^2-4v}-y=& \sqrt{y_L^2-4v_L}-y_L.
\label{2.3.3}
\end{align}
The curves $R_1$ and $R_2$ lie in the closure of the hyperbolic region and intersect only at $U_L.$ The curves $R_1$ and $y^2=4v$ intersect (tangentially) at
\begin{equation}
U_G=(v_G,y_G)=\left(\frac{y_G^2}{4},-
\frac{4v_L}{\sqrt{y_L^2-4v_L}-y_L} \right);
\label{2.3.4}
\end{equation}
the curve $R_1$ and the line $\text{OB}$ intersect at
\begin{equation}
U_H=(v_H,y_H)=\left(\frac{-4\alpha^2 v_L^2+2\alpha \alpha_2 v_L(\sqrt{y_L^2-4v_L}-y_L)}{\alpha^2(\sqrt{y_L^2-4v_L}-y_L)^2-2\alpha_1 v_L(\sqrt{y_L^2-4v_L}-y_L)}, -\frac{\alpha_1v_H}{\alpha^2}-\frac{\alpha_2}{\alpha} \right);
\label{2.3.5}
\end{equation}
the curve $R_2$ and the line $\text{OA}$ intersect at
\begin{equation}
U_C=(v_C,y_C)=\left(\frac{(2\alpha \alpha_1-\alpha^2\sqrt{y_L^2-4v_L}+\alpha^2 y_L)(\sqrt{y_L^2-4v_L}-y_L)}{4\alpha^2-2\alpha_2(\sqrt{y_L^2-4v_L}-y_L)}, -\frac{\alpha_2 v_C}{\alpha^2}-\frac{\alpha_1}{\alpha}\right);
\label{2.3.6}
\end{equation}
the curves $R_2$ and $y^2=4v$ intersect at
\begin{equation}
U_D=(v_D,y_D)=\left(\frac{(y_L-\sqrt{y_L^2-4v_L})^2}{4},y_L-\sqrt{y_L^2-4v_L} \right);
\label{2.3.7}
\end{equation}
the curve $R_2$ and the line $\text{OB}$ intersect at
\begin{equation}
U_E=(v_E,y_E)=\left(\frac{(2\alpha \alpha_2-\alpha^2\sqrt{y_L^2-4v_L}+\alpha^2 y_L)(\sqrt{y_L^2-4v_L}-y_L)}{4\alpha^2-2\alpha_1(\sqrt{y_L^2-4v_L}-y_L)}, -\frac{\alpha_1 v_E}{\alpha^2}-\frac{\alpha_2}{\alpha} \right).
\label{2.3.8}
\end{equation}
The portion of $R_i$ with $v<v_{L}$ corresponds to an admissible rarefaction of the $i$th family, $i=1$ or $2.$ \subsection{\textbf{Shock Curves Through the Left State $U_L$ in the Hyperbolic Region}}\
\\
Using the Rankine-Hugoniot jump conditions,
\begin{align}
s[v]_{\text{jump}}&=\left[\frac{y}{v}\right]_{\text{jump}},
\label{2.4.1} \\
s[y]_{\text{jump}}&=\left[\frac{1}{v}\right]_{\text{jump}},
\label{2.4.2}
\end{align}
we derive
\begin{align}
y&=\frac{v y_L}{2v_L}+\frac{y_L}{2}\pm \frac{1}{2}\frac{(v-v_L)}{v_L}\sqrt{y_L^2-4v_L}; \label{2.4.3}
\end{align}
The choice of sign for $S_1$ and $S_2$ is found by calculating
\begin{align*}
\frac{dy}{dv}|_{u=u_{L}}&=\frac{y_L}{2v_L}\pm \frac{\sqrt{y_L^2-4v_L}}{2v_L}, \\
\frac{dR_1}{dv}|_{u=u_{L}}&=\frac{y_L}{2v_L}-\frac{\sqrt{y_L^2-4v_L}}{2v_L}, \\
\frac{dR_2}{dv}|_{u=u_{L}}&=\frac{y_L}{2v_L}+\frac{\sqrt{y_L^2-4v_L}}{2v_L}.
\end{align*}
Since shock and rarefaction curves have second order contact at $U_L,$ we conclude that the states that can be connected to $U_L$ by a 1-shock or 2-shock lie on the curves
\begin{align}
S_1(v_L,y_L): y&=v(\frac{y_L}{2v_L}-\frac{\sqrt{y_L^2-4v_L}}{2v_L})+\frac{y_L}{2}+\frac{\sqrt{y_L^2-4v_L}}{2} \label{2.4.4}
\end{align}
or
\begin{align}
S_2(v_L,y_L): y&=v(\frac{y_L}{2v_L}+\frac{\sqrt{y_L^2-4v_L}}{2v_L})+\frac{y_L}{2}-\frac{\sqrt{y_L^2-4v_L}}{2} \label{2.4.5}
\end{align}
respectively. The curves $S_1$ and $S_2$ intersect at $U_L.$
\subsection{\textbf{The Lax Shock Admissibility Criterion and Classical Riemann Solutions}}\label{2.5}\
\\
By (\ref{2.4.1})-(\ref{2.4.2})
\begin{align}
s_1 &=\frac{-y_L-\sqrt{y_L^2-4v_L}}{2v v_L},\label{2.5.1} \\
s_2 &=\frac{-y_L+\sqrt{y_L^2-4v_L}}{2v v_L}.\label{2.5.2}
\end{align}
From the eigenvalues (\ref{2.2.2})-(\ref{2.2.3}), we conclude that $\lambda_1(v_L,y_L)>s_1>\lambda_1(v,y)$ and $\lambda_2(v_L,y_L)>s_2>\lambda_2(v,y)$ when $v>v_{L}.$ Therefore the admissible parts of the shock curves consist of points with $v>v_{L}$ and the curves of admissible rarefactions, as stated in the previous section, consist of points with $v<v_{L}$ (if $U_L$ is the state on the left).

The curve $S_1$ and the line $\text{OB}$ intersect at
\begin{equation}
U_H=(v_H,y_H)=\left(\frac{-4\alpha^2 v_L^2+2\alpha \alpha_2 v_L(\sqrt{y_L^2-4v_L}-y_L)}{\alpha^2(\sqrt{y_L^2-4v_L}-y_L)^2-2\alpha_1 v_L(\sqrt{y_L^2-4v_L}-y_L)}, -\frac{\alpha_1v_H}{\alpha^2}-\frac{\alpha_2}{\alpha} \right);
\label{2.5.3}
\end{equation}
the curves $S_2$ and $y^2=4v$ intersect at
\begin{equation}
U_D=(v_D,y_D)=\left(\frac{(y_L-\sqrt{y_L^2-4v_L})^2}{4},y_L-\sqrt{y_L^2-4v_L} \right);
\label{2.5.4}
\end{equation}
the curve $S_2$ and the line $\text{OB}$ intersect at
\begin{equation}
U_E=(v_E,y_E)=\left(\frac{(2\alpha \alpha_2-\alpha^2\sqrt{y_L^2-4v_L}+\alpha^2 y_L)(\sqrt{y_L^2-4v_L}-y_L)}{4\alpha^2-2\alpha_1(\sqrt{y_L^2-4v_L}-y_L)}, -\frac{\alpha_1 v_E}{\alpha^2}-\frac{\alpha_2}{\alpha} \right).
\label{2.5.5}
\end{equation}
Figure \ref{Figure2.1} sketches these curves.
\begin{figure}
\psfrag{1}{$O$}
\psfrag{2}{$A$}
\psfrag{3}{$B$}
\psfrag{4}{$C$}
\psfrag{5}{$F$}
\psfrag{6}{$E$}
\psfrag{7}{$H$}
\psfrag{8}{$U_L$}
\psfrag{9}{$S_1$}
\psfrag{10}{$S_2$}
\psfrag{11}{$R_2$}
\psfrag{12}{$R_1$}
\psfrag{13}{$R_2$}
\psfrag{14}{$D$}
\psfrag{15}{$G$}
\begin{center}
\scalebox{0.7}{
\includegraphics{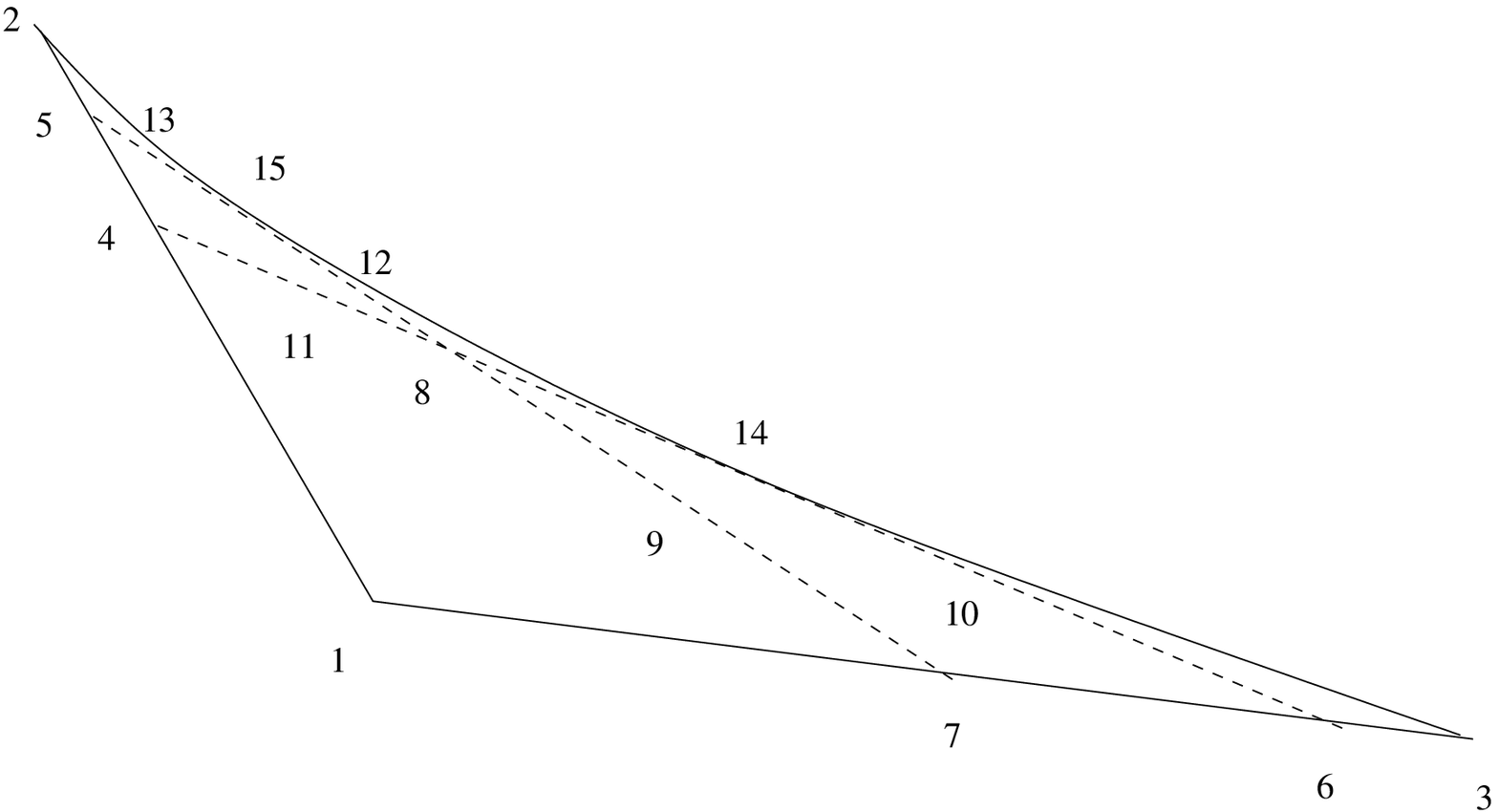}}
\caption{Rarefaction and Shock Curves.}
\label{Figure2.1}
\end{center}
\end{figure}
Using the results of Sections 2.3 and 2.4 and equations (\ref{2.5.1})-(\ref{2.5.2}), we see that in a neighborhood of $U_L$ there exist the usual four regions for the solution of the Riemann problem.

Specifically, we have
\begin{itemize}
  \item Region 1: the unique solution consist of a 1-shock followed by a 2-shock. The region is bounded by $S_1(U_L),$ $S_2(U_L)$ and the line $\text{HE}$ with $v>v_{L}.$
  \item Region 2: the unique solution consist of a 1-rarefaction followed by a 2-rarefaction. We observe that for any $U_L$ the curve $R_1(U_L)$ becomes tangent to $y^2=4v$ at the point $U_G$ identified in equation (\ref{2.3.4}). The (smooth) continuation of this curve is in fact an $R_2$ curve. This curve and the line $\text{OA}$ intersect at
\begin{equation}
U_F=(v_F,y_F)=\left(\frac{-4\alpha^2 v_L^2+2\alpha \alpha_1 v_L(\sqrt{y_L^2-4v_L}-y_L)}{\alpha^2(\sqrt{y_L^2-4v_L}-y_L)^2-2\alpha_2 v_L(\sqrt{y_L^2-4v_L}-y_L)}, -\frac{\alpha_2v_F}{\alpha^2}-\frac{\alpha_1}{\alpha} \right);
\label{2.5.6}
\end{equation}
The region is bounded by $R_2(U_L),$ the line $\text{FC}$ and the curve which begins as $R_1(U_L)$ and continues as $R_2(U_G).$
  \item Region 3: the unique solution consist of a 1-rarefaction followed by a 2-shock. The region is bounded since only the finite interval of $R_1(U_L)$ between $U_L$ and $U_G$ is available for the intermediate state $U_M.$ Furthermore, the interval of admissible points $U_R\in S_2(U_M)$ terminates at a point $U_D(U_M)$ at which the shock speed $s_2=\lambda_1(U_M).$ The upper boundary of Region 3 is the curve $y^2=4v.$ This curve is tangent to $S_2(U_L)$ at the point $U_D.$
  \item Region 4: the unique solution consist of a 1-shock followed by a 2-rarefaction.The region is bounded by the lines $\text{OH},$ $\text{CO},$ $S_1(U_L)$ and the curve $R_2(U_L).$
\end{itemize}
\begin{figure}
\psfrag{1}{$O$}
\psfrag{2}{$A$}
\psfrag{3}{$B$}
\psfrag{4}{$C$}
\psfrag{5}{$F$}
\psfrag{6}{$E$}
\psfrag{7}{$H$}
\psfrag{8}{$U_L$}
\psfrag{9}{$4$}
\psfrag{10}{$1$}
\psfrag{11}{$3$}
\psfrag{12}{$2$}
\psfrag{13}{$5$}
\psfrag{14}{$6$}
\begin{center}
\scalebox{0.7}{
\includegraphics{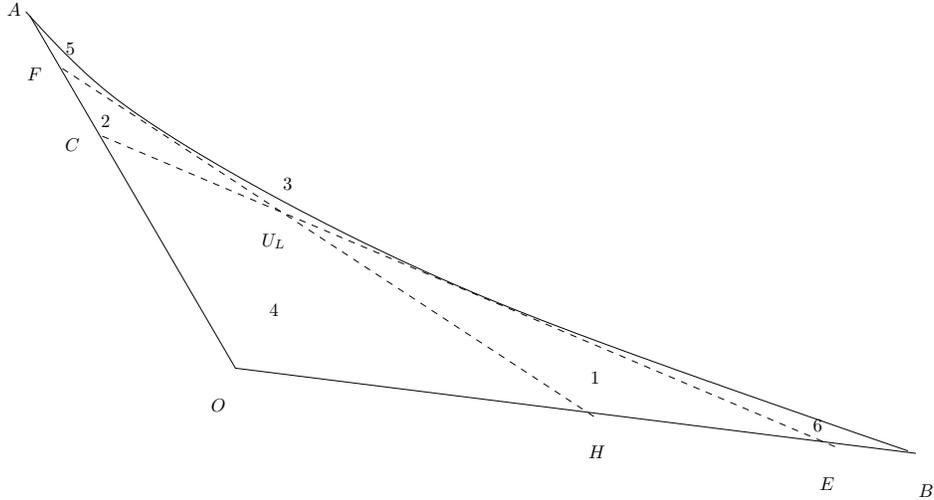}}
\caption{Rarefaction and Shock Curves, Regions}
\label{Figure2.2}
\end{center}
\end{figure}
\subsection{\textbf{Solutions with a Vacuum State}}\label{2.6}\
\\
We now observe that $y^2=4v$ is an invariant curve for (\ref{2.1.4}), and if $(v,y)(x,t)$ is a smooth solution on this curve then $v$ satisfies the equation
\begin{equation}
v_t-(\frac{2}{\sqrt{v}})_x=0.\label{2.6.1}
\end{equation}
Therefore, if $U_R$ is in Region 5 of Figure \ref{Figure2.2}, the solution consists of a 1-rarefaction from $U_L$ to $U_G,$ a rarefaction solution to (\ref{2.6.1}) from $U_G$ to a point $U_{AB}(U_R),$ and a 2-rarefaction from $U_{AB}(U_R)$ to $U_R,$ where $U_{AB}(U_R)$ is the point where $R_2(U_R)$ is tangent to $y^2=4v.$

In region 6, outside these five regions, no classical Riemann solution exists. In the rest of this paper, we show that a solution containing a singular shock can be constructed.
\section{\textbf{The Formal Construction of Singular Shocks}}\label{three}\
\\
This section begins the construction of singular solutions by examining a
self-similar approximation to (\ref{2.1.4}), which provides valuable insight in the GSPT analysis. This will become evident in Section 4.
\subsection{\textbf{Dafermos Regularization}} \label{formal}\
\\
We study systems that approximate (\ref{2.1.4})-(\ref{2.1.5}). Following Dafermos \cite{C.M.Dafermos}, Dafermos and DiPerna \cite{DafermosDiPerna}, and Keyfitz and Kranzer \cite{B.L.KeyfitzandH.C.Kranzer}, we analyze the regularization of $$U_t+F(U)_x=0$$ by a viscous term following Dafermos's approach:
\begin{equation}
\varepsilon t U_{xx}=U_t+F(U)_x.
\label{3.1.1}
\end{equation}
Using $\xi=\frac{x}{t},$ the initial value problem (\ref{3.1.1})-(\ref{2.1.5})
becomes a nonautonomous second-order ODE
\begin{equation}
\varepsilon \frac{d^2 U}{d \xi^2}=\left(DF(U)-\xi I\right)\frac{d U}{d \xi},
\label{3.1.2}
\end{equation}
with boundary conditions
\begin{equation}
U(-\infty)=U_L, \ \ \ U(+\infty)=U_R.
\label{3.1.3}
\end{equation}
Since in the region of interest there are no classical solutions, we seek solutions
that are not uniformly bounded in $\varepsilon$ for $\xi$ near some value $s$.
The following technique, motivated by
Keyfitz and Kranzer \cite{B.L.KeyfitzandH.C.Kranzer}, provides a formal solution.
We develop this and then show in Section 4, following Schecter \cite{Schecter},
that for sufficiently small $\varepsilon>0$, (\ref{3.1.2}) possesses solutions with the
qualitative behavior we predict in this section.

Let
\begin{equation}
U(\xi)=         \begin{pmatrix}
           v(\xi) \\ \\
           y(\xi)
         \end{pmatrix},
\label{3.1.4}
\end{equation}
with $$v(\xi)=\frac{\varepsilon^2 \tilde{u}_2(\frac{\xi-s}{\varepsilon^q})}{(\tilde{u}_1^{2/3}(\frac{\xi-s}{\varepsilon^q})+\varepsilon^{\beta_3})^{3/2}}-\varepsilon^{\beta_4}, \ \ \ y(\xi)=\frac{\varepsilon \tilde{u}_2^2(\frac{\xi-s}{\varepsilon^q})}{(\tilde{u}_1^{16/15}(\frac{\xi-s}{\varepsilon^q})+\varepsilon^{\beta_2})^{3/2}}-\varepsilon^{\beta_1},$$
where $\beta_1>1,$ $\beta_4>\frac{41}{15}$ (the values of $\beta_i,$ $i=1,\ldots,4$ are not unique and are chosen so as to ensure the desired behavior) and define $\eta=\frac{\xi-s}{\varepsilon^q}$.
Then (\ref{3.1.2}) becomes
\begin{equation}
\left\{
    \begin{array}{ll}
    \varepsilon^{3-q}\left(\dfrac{\tilde{u}_2}{(\tilde{u}_1^{2/3}+\varepsilon^{\beta_3})^{3/2}}\right)_{\eta\eta}\ \ &=-(\varepsilon^q \eta+s)\varepsilon^2\left(\dfrac{\tilde{u}_2}{(\tilde{u}_1^{2/3}+\varepsilon^{\beta_3})^{3/2}}\right)_{\eta}\\
    &+\varepsilon^{-1}
    \left(\dfrac{\left[{ \tilde{u}_2^2(\frac{\xi-s}{\varepsilon^q})}-\varepsilon^{\beta_1-1}(\tilde{u}_1^{16/15}(\frac{\xi-s}{\varepsilon^q})+\varepsilon^{\beta_2})^{3/2}\right](\tilde{u}_1^{2/3}(\frac{\xi-s}{\varepsilon^q})+\varepsilon^{\beta_3})^{3/2}}{\left[{ \tilde{u}_2(\frac{\xi-s}{\varepsilon^q})}-\varepsilon^{\beta_4-2}(\tilde{u}_1^{2/3}(\dfrac{\xi-s}{\varepsilon^q})+\varepsilon^{\beta_3})^{3/2}\right](\tilde{u}_1^{16/15}(\frac{\xi-s}{\varepsilon^q})+\varepsilon^{\beta_2})^{3/2}}\right)_{\eta}, \\ \\
\varepsilon^{2-q}\left(\dfrac{\tilde{u}_2^2}{(\tilde{u}_1^{16/15}+\varepsilon^{\beta_2})^{3/2}}\right)_{\eta\eta}&=-(\varepsilon^q \eta+s)\varepsilon\left(\dfrac{\tilde{u}_2^2}{(\tilde{u}_1^{16/15}+\varepsilon^{\beta_2})^{3/2}}\right)_\eta\\
&+\varepsilon^{-2}\left(\dfrac{(\tilde{u}_1^{2/3}(\frac{\xi-s}{\varepsilon^q})+\varepsilon^{\beta_3})^{3/2}}{{ \tilde{u}_2(\frac{\xi-s}{\varepsilon^q})}-\varepsilon^{\beta_4-2}(\tilde{u}_1^{2/3}(\frac{\xi-s}{\varepsilon^q})+\varepsilon^{\beta_3})^{3/2}}\right)_\eta.
    \end{array}
  \right.
\label{3.1.5}
\end{equation}
We balance at least two terms in each equation, so that nontrivial solutions can be found. Thus we set $3-q=-1$ in the first equation, and $2-q=-2$ in the second.
This gives $q=4$ and hence
\begin{equation}
\left\{
    \begin{array}{ll}
    \left(\frac{\tilde{u}_2}{(\tilde{u}_1^{2/3}+\varepsilon^{\beta_3})^{3/2}}\right)_{\eta\eta}\ \ &=-(\varepsilon^q \eta+s)\varepsilon^3\left(\frac{\tilde{u}_2}{(\tilde{u}_1^{2/3}+\varepsilon^{\beta_3})^{3/2}}\right)_{\eta}\\
    &+
    \left(\dfrac{\left[{ \tilde{u}_2^2(\frac{\xi-s}{\varepsilon^q})}-\varepsilon^{\beta_1-1}(\tilde{u}_1^{16/15}(\frac{\xi-s}{\varepsilon^q})+\varepsilon^{\beta_2})^{3/2}\right](\tilde{u}_1^{2/3}(\frac{\xi-s}{\varepsilon^q})+\varepsilon^{\beta_3})^{3/2}}{\left[{ \tilde{u}_2(\frac{\xi-s}{\varepsilon^q})}-\varepsilon^{\beta_4-2}(\tilde{u}_1^{2/3}(\dfrac{\xi-s}{\varepsilon^q})+\varepsilon^{\beta_3})^{3/2}\right](\tilde{u}_1^{16/15}(\frac{\xi-s}{\varepsilon^q})+\varepsilon^{\beta_2})^{3/2}}\right)_{\eta}, \\ \\
\left(\frac{\tilde{u}_2^2}{(\tilde{u}_1^{16/15}+\varepsilon^{\beta_2})^{3/2}}\right)_{\eta\eta}&=-(\varepsilon^q \eta+s)\varepsilon^3\left(\frac{\tilde{u}_2^2}{(\tilde{u}_1^{16/15}+\varepsilon^{\beta_2})^{3/2}}\right)_\eta\\
&+\left(\dfrac{(\tilde{u}_1^{2/3}(\frac{\xi-s}{\varepsilon^q})+\varepsilon^{\beta_3})^{3/2}}{{ \tilde{u}_2(\frac{\xi-s}{\varepsilon^q})}-\varepsilon^{\beta_4-2}(\tilde{u}_1^{2/3}(\frac{\xi-s}{\varepsilon^q})+\varepsilon^{\beta_3})^{3/2}}\right)_\eta.
    \end{array}
  \right.
\label{3.1.6}
\end{equation}
The singular region is narrower than a standard shock profile.

When we expand $\tilde{u}_1$,  $\tilde{u}_2$ as series in $\varepsilon$
$$\tilde{u}_1=y_1(\eta)+o(1), \ \ \tilde{u}_2=y_2(\eta)+o(1),$$
we obtain
\begin{equation}
\left\{
    \begin{array}{ll}
    \left(\frac{y_2}{y_1}\right)_{\eta\eta}\ \ \ =\left(\frac{y_2}{y_1^{3/5}}\right)_{\eta}, \\  \\
\left(\frac{y_2^2}{y_1^{8/5}}\right)_{\eta\eta}=\left(\frac{y_1}{y_2}\right)_\eta.
    \end{array}
  \right.
\label{3.1.7}
\end{equation}
We note that from (\ref{3.1.4}) we must have $y_1, y_2 \rightarrow 0$ as $|\eta|\rightarrow \infty,$ and $\dfrac{y_2}{y_1^{3/5}}, \dfrac{y_1}{y_2}\rightarrow 0$ as $\eta \rightarrow \infty.$ Assuming that the singular behavior is restricted to a neighborhood of $\xi=s$ we also have $\left(\dfrac{y_2}{y_1}\right)_\eta, \left(\dfrac{y_2^2}{y_1^{8/5}}\right)_{\eta} \rightarrow 0$ as $\eta \rightarrow \infty.$ We integrate (\ref{3.1.7}) once, and now focus attention on solutions of
\begin{equation}
\left\{
    \begin{array}{ll}
    \frac{d y_1}{d\eta}=\frac{5}{2}\left(\frac{y_1^{18/5}}{y_2^3}-2y_1^{7/5}\right),\\ \\
\frac{d y_2}{d\eta}=\frac{5}{2}\frac{y_1^{13/5}}{y_2^2}-4y_2y_1^{2/5},
    \end{array}
  \right.
\label{3.1.8}
\end{equation}
which approach $(0,0)$ as $|\eta|\rightarrow \infty$.

The phase portrait of the 2-dimensional system (\ref{3.1.8}) is shown in Figure \ref{Figure3.1}.
The origin is the unique equilibrium.
$y_2=2^{1/3}y_1^{11/15}$ is an invariant parabola (in $(v,y)$ coordinates this curve is $y^2=2v$). The line $y_2=0$ corresponds to the point $(0,0)$ in $(v,y)$ coordinates. The homoclinic orbits, which are of greatest interest
to us, are  solutions $(y_1(\eta),y_2(\eta))$ which can be determined
uniquely by setting $y_1(0)>0$, $y_2(0)>0.$
\begin{figure}
\psfrag{1}{$y_1$}
\psfrag{2}{$y_2$}
\psfrag{3}{$y_2=2^{1/3}y_1^{11/15}$}
\begin{center}
\scalebox{0.8}{
\includegraphics{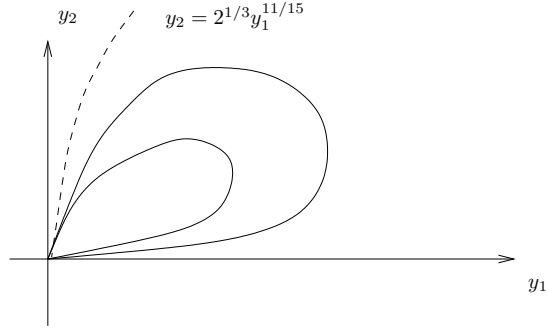}}
\caption{Integral Curves and Orbits of (\ref{3.1.8}).}
\label{Figure3.1}
\end{center}
\end{figure}
We will need to know the asymptotic behavior of $Y=(y_1,y_2)$ as
$|\eta|\rightarrow \infty$.
Writing
\begin{align}
y_1&=\frac{c}{|\eta|^p}+{\mathcal O}\left(\frac{1}{|\eta|^{p+1}}\right),\label{3.1.9}\\
y_2&=\frac{d}{|\eta|^r}+{\mathcal O}\left(\frac{1}{|\eta|^{r+1}}\right),
\label{3.1.10}
\end{align}
we substitute (\ref{3.1.9})-(\ref{3.1.10}) into (\ref{3.1.8}) and then solve for $c,$ $d,$ $p$ and $r$ to obtain
$$c=\left(\frac{2}{3}\right)^{5/2}, \ \ d=3^{1/3}\left(\frac{2}{3}\right)^{13/6}, \ \ r=\frac{11}{6}, \ \ p=\frac{5}{2}$$ as $\eta\rightarrow \infty.$ 
We also have
\begin{equation}
\label{be}
y_2\approx 2^{1/3}y_1^{11/15}
\end{equation}
as $Y\to 0$.
This describes the asymptotic behavior of $Y$ as $Y\to 0$.
Therefore the homoclinic orbits are tangent to the invariant parabola $y_2\approx 2^{1/3}y_1^{11/15}$ at one end. In addition we have $c=d=0$ as $\eta\rightarrow -\infty.$

The singular solution (\ref{3.1.4}), has its essential support in a layer of width $|\xi-s|=
O(\varepsilon^q)$ with $q>1,$ and tends to zero away from $\xi=s$.
As in Keyfitz and Kranzer \cite{B.L.KeyfitzandH.C.Kranzer} we embed the singular shock in
a shock profile of the usual type: a solution $\bar{U}(\tau)=\bar{U}(\frac{\xi-s}{\varepsilon})$
which is bounded in the layer $\varepsilon^{q}<|\xi-s|<\varepsilon$, has an expansion
\begin{equation}
\bar{U}=\bar{U}_0+o(1),
\label{3.1.11}
\end{equation}
and whose derivatives are $O(\varepsilon^{-1})$.

Writing (\ref{3.1.2}) in terms of $\tau=\frac{\xi-s}{\varepsilon}$ we have \begin{equation}
\frac{d}{d\tau}\left(\frac{d\bar{U}}{d \tau}-F(\bar{U})+s\bar{U}\right)=-\varepsilon \tau\frac{d \bar{U}}{d\tau}.
\label{3.1.12}
\end{equation}
Using the expansion (\ref{3.1.11}) we have
\begin{equation}
\frac{d}{d\tau}\left(\frac{d\bar{U}_0}{d \tau}-F(\bar{U}_0)+s\bar{U}_0\right)=0,
\label{3.1.13}
\end{equation}
in each separate interval $\tau<0$ or $\tau>0$ outside the boundary layer.
Hence, we may write
\begin{equation}
\frac{d\bar{U}_0}{d \tau}-F(\bar{U}_0)+s\bar{U}_0=k_{\mp}.
\label{3.1.14}
\end{equation}
On the other hand, integrating (\ref{3.1.12}) over an interval surrounding $\tau = 0$
(the boundary layer), we obtain
\begin{equation}
\left[\frac{d\bar{U}}{d \tau}-F(\bar{U})+s\bar{U}\right]_{\tau<0}^{\tau>0}=-\varepsilon \int_{\tau<0}^{\tau>0}\tau\frac{d\bar{U}}{d\tau}d\tau.
\label{3.1.15}
\end{equation}
Now, from \eqref{3.1.4}, $$\bar{U}(\tau)=
                     U(\xi),$$
   and we change the variable to $\eta = \tau/\varepsilon^3$
  in (\ref{3.1.15}), which yields
\begin{align}
k_{+}-k_{-}&=\lim_{\varepsilon\rightarrow 0}\left\{-\varepsilon \int\varepsilon^3 \eta\left(
                                                          \begin{array}{c}
                                                            \frac{dv}{d\eta} \\ \\
                                                            \frac{dy}{d\eta} \\
                                                          \end{array}
                                                        \right)d\eta \ \right\}=\lim_{\varepsilon\rightarrow 0}
    \begin{pmatrix}
      -\varepsilon^6 \int\eta \frac{d}{d\eta}\left(\frac{\tilde{u}_2}{(\tilde{u}_1^{2/3}+\varepsilon^{\beta_3})^{3/2}}\right) d\eta\\ \\
      -\varepsilon^{5} \int\eta \frac{d}{d\eta}\left(\frac{\tilde{u}_2^2}{(\tilde{u}_1^{16/15}+\varepsilon^{\beta_2})^{3/2}}\right) d\eta
    \end{pmatrix}\nonumber \\
&=\lim_{\varepsilon\rightarrow 0}
    \begin{pmatrix}
      -\varepsilon^6 \int\eta \frac{d}{d\eta}\left(\frac{y_2}{(y_1^{2/3}+\varepsilon^{\beta_3})^{3/2}}\right) d\eta\\ \\
      -\varepsilon^{5} \int\eta \frac{d}{d\eta}\left(\frac{y_2^2}{(y_1^{16/15}+\varepsilon^{\beta_2})^{3/2}}\right) d\eta
    \end{pmatrix}\nonumber\\
    &=\lim_{\varepsilon\rightarrow 0}
    \begin{pmatrix}
      -\varepsilon^6 \int_{\text{finite} \ \eta}\eta \frac{d}{d\eta}\left(\frac{y_2}{(y_1^{2/3}+\varepsilon^{\beta_3})^{3/2}}\right) d\eta-\varepsilon^6 \int_{\text{infinite} \ \eta}\eta \frac{d}{d\eta}\left(\frac{y_2}{(y_1^{2/3}+\varepsilon^{\beta_3})^{3/2}}\right) d\eta\\ \\
      -\varepsilon^{5} \int_{\text{finite} \ \eta}\eta \frac{d}{d\eta}\left(\frac{y_2^2}{(y_1^{16/15}+\varepsilon^{\beta_2})^{3/2}}\right) d\eta-\varepsilon^{5} \int_{\text{infinite} \ \eta}\eta \frac{d}{d\eta}\left(\frac{y_2^2}{(y_1^{16/15}+\varepsilon^{\beta_2})^{3/2}}\right) d\eta
    \end{pmatrix}.
\label{3.1.16}
\end{align}
When $\eta$ is finite we notice that for values of $y_1$ and $y_2$ away from the origin $y$ and $v$ are close to zero, therefore we can focus on the case of $y_1, y_2\rightarrow 0.$ If $v\rightarrow 0$ and $y\rightarrow \infty$ then $\varepsilon^4 v\rightarrow 0.$ If $v, y\rightarrow \infty$ then by (\ref{3.1.7}) $$\frac{d}{d\eta}\left(\frac{y_2^2}{(y_1^{16/15}+\varepsilon^{\beta_2})^{3/2}}\right)=\frac{\varepsilon^2}{v+\varepsilon^{\beta_4}}.$$ In addition $v\simeq \varepsilon^k\sqrt{y}$ where $-1<k<2.5$ and $$\varepsilon^6\frac{d}{d\eta}\left(\frac{y_2}{(y_1^{2/3}+\varepsilon^{\beta_3})^{3/2}}\right)=\varepsilon^7 \frac{y}{v}=\frac{\varepsilon^6}{\varepsilon^2}\frac{dv}{d\eta}\approx\frac{\varepsilon^4\varepsilon^k}{2\sqrt{y}}\frac{dy}{d\eta}=\frac{\varepsilon^5\varepsilon^k}{2\sqrt{y}}\frac{d}{d\eta}\left(\frac{y_2^2}{(y_1^{16/15}+\varepsilon^{\beta_2})^{3/2}}\right)$$
so either after short calculations or integration by parts all cases give $$\lim_{\varepsilon\rightarrow 0} \varepsilon^6\int_{\text{finite} \ \eta}\eta \frac{d}{d\eta}\left(\frac{y_2}{(y_1^{2/3}+\varepsilon^{\beta_3})^{3/2}}\right) d\eta=0.$$
The interesting behavior which will give us the generalized Rankine-Hugoniot condition emerges as $\eta \rightarrow \infty.$ We use (\ref{3.1.9})-(\ref{3.1.10}), ignoring the constants $c$ and $d,$ without loss of generality, and letting $$\frac{1}{\eta^{5/3}}=\varepsilon^{\beta_3} \tan^2\theta$$ to get
\begin{align*}
-\varepsilon^6 \int_{\text{infinite} \ \eta}\eta \frac{d}{d\eta}\left(\frac{y_2}{(y_1^{2/3}+\varepsilon^{\beta_3})^{3/2}}\right) d\eta&=-\varepsilon^6\frac{\frac{\eta}{\eta^{11/6}}}{\left(\frac{1}{\eta^{5/3}}+\varepsilon^{\beta_3}\right)^{3/2}}\big|_{\text{infinite} \ \eta}+\varepsilon^6\int_{\text{infinite} \ \eta} \frac{\frac{1}{\eta^{11/6}}}{\left(\frac{1}{\eta^{5/3}}+\varepsilon^{\beta_3}\right)^{3/2}} \ d\eta\\
&=\varepsilon^{6-\beta_3}\sin \theta_0 \cos^2 \theta_0+ \frac{6}{5}\cdot\varepsilon^{6-\beta_3}\int_0^{\theta_0} \cos \theta \ d\theta\simeq\varepsilon^{6-\beta_3},
\end{align*}
for some $\theta_0.$
On the other hand, if we let $$\frac{1}{\eta^{8/3}}=\varepsilon^{\beta_2}\tan^2 \theta$$ we get
\begin{align*}
-\varepsilon^{5} \int_{\text{infinite} \ \eta}\eta \frac{d}{d\eta}\left(\frac{y_2^2}{(y_1^{16/15}+\varepsilon^{\beta_2})^{3/2}}\right) d\eta&=-\varepsilon^5 \frac{\frac{\eta}{\eta^{11/3}}}{\left(\frac{1}{\eta^{8/3}}+\varepsilon^{\beta_2}\right)^{3/2}} \big|_{\text{infinite} \ \eta}+
\varepsilon^5\int_{\text{infinite} \ \eta} \frac{\frac{1}{\eta^{11/3}}}{\left(\frac{1}{\eta^{8/3}}+\varepsilon^{\beta_2}\right)^{3/2}} \ d\eta\\
&=\varepsilon^{5-\frac{\beta_2}{2}}\sin^2\theta_1\cos \theta_1+\frac{3}{4}\cdot \varepsilon^{5-\frac{\beta_2}{2}}\int_0^{\theta_1}\sin \theta \ d\theta\simeq \varepsilon^{5-\frac{\beta_2}{2}},
\end{align*}
for some $\theta_1.$
\subsection{\textbf{Possible Cases.}}\
\begin{enumerate}
  \item If $\beta_3=6,$ $\beta_2<10$ then
\begin{equation}
k_{+}-k_{-}=\left(
  \begin{array}{c}
    \ast \\
    0 \\
  \end{array}
\right).
\nonumber
\end{equation}
By (\ref{3.1.3}), we have $\bar{U}_0(-\infty)=U_L$,
$ \bar{U}_0(+\infty)=U_R$ and $\frac{d\bar{U}_0}{d\tau}(\pm\infty)=0$.
 Therefore, from (\ref{3.1.14}) we get the generalized Rankine-Hugoniot condition for singular shocks:
\begin{align}
s_{\text{singular}}(U_L,U_R)&=s=\frac{F_2(U_L)-F_2(U_R)}{y_{L}-y_{R}}, \label{3.2.1} \\
0<k&=F_1(U_L)-F_1(U_R)-s(v_{L}-v_{R}).
\label{3.2.2}
\end{align}
We notice that we have a deficit in the first component. This does not agree with Mazzotti \cite{Mazzotti2}. In addition if we check for the overcompressive region $$\lambda_2(v,y)<s<\lambda_1(v_L,y_L)$$ we see that region 6 is overcompressive but the slope of the curve $s=\lambda_1(v_L,y_L)$ is negative. Therefore the region does not look like the required one, which should cover all possible solutions of the Riemann problem in the plane.
  \item If $\beta_3<6,$ $\beta_2<10$ then we get the Rankine-Hugoniot confition for both components but this does not give a singular shock.
  \item If $\beta_3=6,$ $\beta_2=10$ then
  \begin{equation}
k_{+}-k_{-}=\left(
  \begin{array}{c}
    \ast \\
    \ast \\
  \end{array}
\right).
\nonumber
\end{equation}
This means we have a deficit for both components.  As $\eta \rightarrow \infty$ the solution (\ref{3.1.4}) in this case behaves like $$v=\frac{\varepsilon^2 \cdot \frac{1}{\eta^{11/6}}}{\left(\frac{1}{\eta^{5/3}}+\varepsilon^6\right)^{3/2}}-\varepsilon^{\beta_4}, \ \ y=\frac{\varepsilon \cdot \frac{1}{\eta^{11/3}}}{\left(\frac{1}{\eta^{8/3}}+\varepsilon^{10}\right)^{3/2}}-\varepsilon^{\beta_1}.$$
Let $$\frac{1}{\eta^{1/3}}= \tan\theta,$$ then as $\theta\rightarrow 0$
$$v=\frac{\varepsilon^2\tan^{11/2} \theta}{(\tan^5\theta+\varepsilon^6)^{3/2}}-\varepsilon^{\beta_4}, \ \ y=\frac{\varepsilon\tan^{11}\theta}{(\tan^8 \theta+\varepsilon^{10})^{3/2}}-\varepsilon^{\beta_1}.$$ As $\varepsilon\rightarrow 0$ one should expect $y-v$ to have a bigger maximum value than $v-y$ (as we will see in Figure \ref{Figure4.3}). However, this is not the case here. In addition this would not agree with Mazzotti \cite{Mazzotti2}.
\item If $\beta_3\leqslant 5,$ $\beta_2=10$ then to see how this solution behaves for small $\varepsilon$
as $\eta\rightarrow\infty$ we may let $$\frac{1}{\eta^{1/3}}=\varepsilon \tan\theta,$$ 
$$v=\frac{\tan^{11/2} \theta}{(\tan^5\theta+1)^{3/2}}-\varepsilon^{\beta_4}.$$
$v$ remains bounded but since $y$ is unbounded one should expect $v$ to be unbounded as well by (\ref{be}). 
\item If $5<\beta_3<6,$ $\beta_2=10$ then 
$$k_{+}-k_{-}=\left(
                        \begin{array}{c}
                          0 \\
                          k \\
                        \end{array}
                      \right),$$
where $$k=-\lim_{\varepsilon\rightarrow 0} \varepsilon^{5} \int\eta \frac{d}{d\eta}\left(\frac{y_2^2}{(y_1^{16/15}+\varepsilon^{\beta_2})^{3/2}}\right) d\eta,$$ defined uniquely by each orbit.
Finally, from (\ref{3.1.14}) we get the generalized Rankine-Hugoniot condition for singular shocks:
\begin{align}
s_{\text{singular}}(U_L,U_R)&=s=\frac{F_1(U_L)-F_1(U_R)}{v_{L}-v_{R}}, \label{3.2.3} \\
0<k&=F_2(U_L)-F_2(U_R)-s(y_{L}-y_{R}).
\label{3.2.4}
\end{align}
The restriction on the sign of $k$ is consistent with having $U_R$ in region 6
with respect to $U_L$.

We now introduce two curves, as shown in Figure \ref{Figure3.2}, namely $J_5$
and $J_6$ determined by
$$s_{\text{singular}}(U_L,U)=\lambda_1(U_L)$$ and
$$s_{\text{singular}}(U_L,U)=\lambda_2(U),$$ respectively.
We find
\begin{align}
J_5: y=\frac{y_L}{v_L}v+v(v-v_L)\cdot (\frac{-y_L-\sqrt{y_L^2-4v_L}}{2v_L^2})
\label{3.2.5}
\end{align}
The curve
$J_5$ passes through the point $U_L$ and intersects $y^2=4v$ at a point $U_D.$
The second curve is $J_6$, given by
\begin{align}
J_6: y=\frac{vy_L(2v-v_L)+v^2y_L}{2v_L(2v-v_L)}+\frac{(v-v_L)}{2v_L(2v-v_L)}\sqrt{\left( v y_L-4\frac{v_L^2}{y_L}\right)^2+4v_L^3\frac{(y_L^2-4v_L)}{y_L^2}}
\label{3.2.6}
\end{align}
The curve $J_6$ passes through the point $U_L$ and does not intersect $y^2=4v.$ $\beta_3$ is chosen such that $y$ is unbounded as $y_1, y_2\rightarrow 0$ and $v$ passes from a neighborhood of $0$ (where the variables $u_1,$ $u_2$ of (\ref{2.1.1}) become singular) before becoming unbounded. In addition $y-v$ has a bigger maximum value than $v-y.$ 
\end{enumerate}
\begin{figure}
\psfrag{1}{$O$}
\psfrag{2}{$A$}
\psfrag{3}{$B$}
\psfrag{4}{$C$}
\psfrag{11}{$D$}
\psfrag{8}{$U_L$}
\psfrag{6}{$E$}
\psfrag{12}{$J_6$}
\psfrag{13}{$J_5$}
\begin{center}
\scalebox{0.7}{
\includegraphics{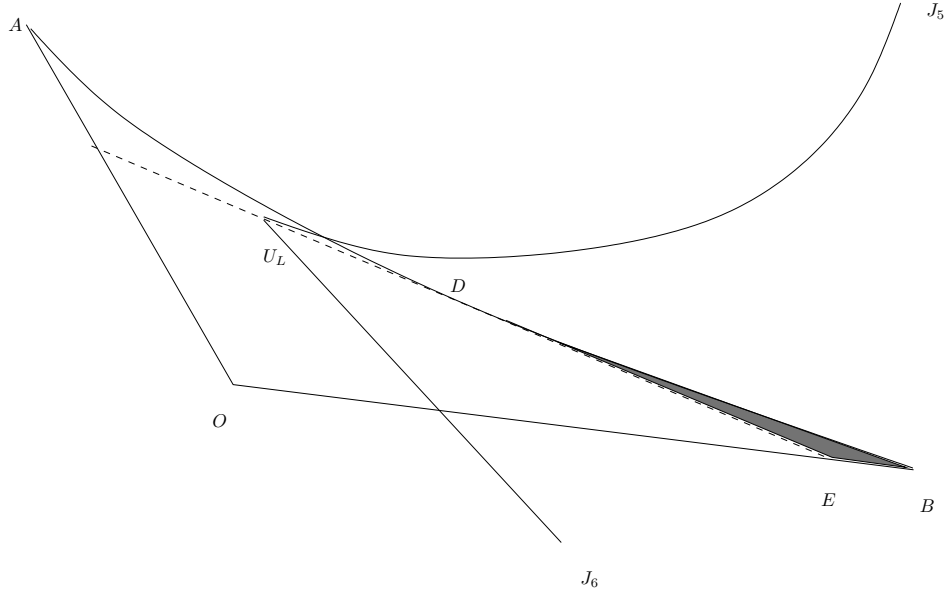}}
\caption{Regions of Singular Shocks, Additional Curves.}
\label{Figure3.2}
\end{center}
\end{figure}
We conclude that the forth case agrees with Mazzotti \cite{Mazzotti2} whereas all other cases fail.

In the remainder of this paper we  show existence of  self-similar singular shock
solutions to (\ref{3.1.1}).
Our main result is the following theorem.
\begin{theorem} \label{thm3.1}
In the system of conservation laws (\ref{2.1.4}) with Riemann data (\ref{2.1.5}),
assume that $U_R$ is in the interior of region 6 with respect to $U_L$, so that with
\begin{equation}
s_{\text{singular}}(U_L,U_R)\equiv
\frac{F_1(U_L)-F_1(U_R)}{v_{L}-v_{R}}, \label{3.2.7}
\end{equation}
we have
\begin{equation}
0<k=F_2(U_L)-F_2(U_R)-s_{\text{singular}}(y_{L}-y_{R})\,,
\label{3.2.8}
\end{equation}
and the strict inequalities
\begin{enumerate}
  \item $s_{\text{singular}}(U_L,U_R)<\lambda_1(U_L)$
  \item $\lambda_2(U_R)<s_{\text{singular}}(U_L,U_R)$
\end{enumerate}
hold.
Then there exists a singular shock connecting $U_L$ and $U_R$ passing from points very close to the $y$-axis (thus the chromatography model (\ref{2.1.1}) exhibits singular shocks); that is, a solution
$U_\varepsilon$ of (\ref{3.1.2})-(\ref{3.1.3}) which becomes
unbounded as $\varepsilon\rightarrow 0$.
\end{theorem}
\subsection{\textbf{Remarks.}}\
\\
Since we are only interested in the curvilinear triangle $OAB,$ proving existence of a self-similar approximate solution in region 6 -- in which the Riemann problem is solved by a strictly overcompressive singular shock alone -- completes the list of solutions in regions 1-5, given in Sections \ref{2.5} and \ref{2.6}.
 
In Section \ref{four}, we prove Theorem \ref{thm3.1} by showing existence of solutions
to equations \eqref{3.1.2} and \eqref{3.1.3} for small $\varepsilon$.
We use the approach of Schecter \cite{Schecter}, which proceeds by modifying
GSP theory \cite{F,J} to take into account that normal hyperbolicity fails
in parts of the construction.
A method for handling loss of normal hyperbolicity, known as ``blowing up",
was developed by Krupa and Szmolyan,
\cite{KS}, and applied by Schecter. Strict overcompressibility is needed, as will be seen, to carry out the construction.
\section{\textbf{Existence of Approximations to Singular Solutions}}
\label{four}\ \\
We use GSPT to prove Theorem \ref{thm3.1} by showing that self-similar regularized solutions exist for sufficiently small $\varepsilon>0.$ The approach was laid out by Schecter \cite{Schecter} and was also employed by Keyfitz and Tsikkou \cite{KT}.

Basically, the situation described in Section \ref{three} consist of an ``outer'' part (which includes the two constant states $U_L$ and $U_R$) and an ``inner'' part (the scaled homoclinic orbit) with no indication how to connect them. The treatment following \eqref{3.1.11} did not prove that a solution exists, but just simply suggests a mechanism whereby the two parts of the solution could be connected. This is corrected by Geometric singular perturbation theory (GSPT), using the theory of
dynamical systems to prove that smooth systems, under the appropriate nondegeneracy conditions,
do possess connecting orbits, and even that these orbits are unique.
GSPT was developed by Fenichel \cite{F} (see also the exposition by Jones
\cite{J}), and despite many efforts, points where normal hyperbolicity breaks down, as in our case, remained a major obstacle to the geometric theory. (A flow is normally hyperbolic with respect to an invariant manifold if any manifold transverse to the flow can be factored into stable and unstable directions. More precisely, if the system is linearized at a point on the invariant manifold, then only the eigenvalues with eigenvectors tangent to the invariant manifold have zero real part.) Krupa and Szmolyan \cite{KS} applied their technique of ``blowing up'' to some
examples, but it was Schecter who showed how it could also apply to the system
\eqref{3.1.2}. The insight of GSPT is that one can study the systems when $\varepsilon=0$ and then piece the information together to prove the existence of a genuine orbit when $\varepsilon >0$. Within the framework of GSPT, and following Schecter, we find a way to connect the homoclinic orbit produced in the previous section with the skeleton that joins $U_L$ and $U_R.$

The objective of this section is to apply the theory of dynamical systems to prove existence of an orbit when $\varepsilon>0.$ The important tool is the {\em Exchange Lemma\/} of Jones and Kopell \cite {JK}, and an extension called by
Schecter \cite{Schecter} the {\em Corner Lemma\/}. GSPT approach replaces a dynamical problem, here
\eqref{3.1.2} and \eqref{3.1.3}, in which a singular limit occurs, with a higher-dimensional
dynamical system in which $\{\varepsilon = 0\}$ is merely a subspace, and behavior
near that subspace can be determined by continuity if the hypotheses of the Exchange
and Corner Lemmas are satisfied. We will describe the pieces of the solution in the singular limit and verify the nondegeneracy hypotheses needed to carry out the perturbation. As could be seen already in Section \ref{formal}, some rescaling of the variables is needed to exhibit any of the dynamics on the fast time scale. In addition, the technique of ``blowing up'' which involves a change of variables to desingularize the invariant manifold will be used to reveal essential information about the flow and gain additional hyperbolicity. 
\subsection{\textbf{Creating the Dynamical Problem}}\
\\
We start from (\ref{3.1.2})-(\ref{3.1.3}) introducing  $V=\left(
                        \begin{array}{c}
                          v_1 \\
                          v_2 \\
                        \end{array}
                      \right)=\varepsilon \frac{dU}{d\xi}$, and
$\theta=\xi-s_{\text{singular}}.$ It is also convenient to treat $\xi$ as a state variable. This increases the dimension, but yields an
autonomous system.
Therefore the problem in the original self-similar variable (the slow time $\theta$) is
\begin{equation}
\begin{aligned}
\varepsilon\frac{dv}{d\theta}&=v_1, \\ \varepsilon\frac{dy}{d\theta}&=v_2, \\
\varepsilon\frac{dv_1}{d\theta}&=\frac{v_2}{v}-\frac{y v_1}{v^2}-\xi v_1, \\ \varepsilon\frac{dv_2}{d\theta}&=-\frac{v_1}{v^2}-\xi v_2, \\
\frac{d\xi}{d\theta}&=1.
\end{aligned}
\label{4.1.1}
\end{equation}
As written, this is singular as $\varepsilon \to 0$.
Replacing $\theta$ with $\tau$, where
$\theta=\varepsilon \tau$, we will work in the fast time system
\begin{equation}
\begin{aligned}
\frac{dv}{d\tau}&=v_1, \\ \frac{dy}{d\tau}&=v_2, \\
\frac{dv_1}{d\tau}&=\frac{v_2}{v}-\frac{yv_1}{v^2}-\xi v_1, \\ \frac{dv_2}{d\tau}&=-\frac{v_1}{v^2}-\xi v_2, \\
\frac{d\xi}{d\tau}&=\varepsilon,
\end{aligned}
\label{4.1.2}
\end{equation}
We note that in this problem ``slow'' and ``fast'' do {\em not} correspond to
``outer'' and ``inner''.
In fact, we will need an inner, faster time variable ($\eta =\tau/\varepsilon^3$) to describe the inner solution, as done formally in the previous section.

The boundary conditions are
\begin{equation}
(U, V, \xi)(-\infty)=(U_L, 0, -\infty), \quad (U, V, \xi)(+\infty)=(U_R, 0, +\infty).
\label{4.1.3}
\end{equation}
We now let $\varepsilon=0$ in (\ref{4.1.2}),
noting that \eqref{4.1.2} is now a regularly perturbed problem.
With $\xi = \textit{const.}$ for all solutions, the states $V=0$ are all
equilibria, and they are the only equilibria.

Using the eigenvalues (\ref{2.2.2})-(\ref{2.2.3}) we identify two subsets of $S$:
For $\delta>0$, we define $3$-dimensional manifolds
\begin{equation*}
\begin{aligned}
S_0&=\{(U,V,\xi): \|U\|\leq\frac{1}{\delta}, \ \ V=0, \ \ \text{and} \ \ \xi\leq \lambda_1(U)-\delta\}, \\
S_2&=\{(U,V,\xi): \|U\|\leq\frac{1}{\delta}, \ \ V=0,  \ \ \text{and} \ \ \lambda_2(U)+\delta\leq\xi\}.
\end{aligned}
\end{equation*}
which are normally hyperbolic
since the lines $\xi=\lambda_1(U), \ \ \xi=\lambda_2(U)$ are not included in the sets $S_i$.
In fact, if we linearize (\ref{4.1.2}) and set $\varepsilon=0, \ \ V=0$,
 there are $3$ eigenvalues of zero, with a full set of  eigenvectors in the space of equilibria.
 The remaining eigenvalues, $-\xi+\lambda_1(U)$ and $-\xi+\lambda_2(U)$,
 are real and nonzero.
 In $S_0$, both are positive, so there is  an unstable manifold of dimension $2$;
 and in $S_2$  a stable manifold of dimension $2$.
The boundary value $(U_L,0,-\infty)$ is an $\alpha$-limit of points in $S_0$, and
$(U_R,0,+\infty)$ an $\omega$-limit in $S_2$.

By Fenichel \cite{F}, and as stated in Schecter \cite{Schecter},
a system with normally hyperbolic manifolds of equilibria has perturbed
normally hyperbolic invariant manifolds nearby.
That is the case here: For $\varepsilon>0$ and near $0$, by Fenichel theory \cite{F},
the system (\ref{4.1.2}) has normally hyperbolic invariant manifolds near
each $S_i$. Since the $3$-dimensional space $S\equiv \{(U,V,\xi): V=0\}$ is invariant under (\ref{4.1.2}) for every $\varepsilon,$ the perturbed manifolds may be taken to be the $S_i$ themselves.

For a given $U_L$, we define the $1$-dimensional invariant set
\begin{equation}
S_0(U_L)=\{(U,V,\xi): U=U_L, \ \ V=0, \ \ \xi< \lambda_1(U_L)\}\,.
\nonumber
\end{equation}
The line $S_0(U_L)$ possesses a $3$-dimensional unstable manifold
$W_{\varepsilon}^{u}(S_0(U_L))$, the perturbation of
\begin{equation}
W_{0}^{u}(S_0(U_L))=\{(U,V,\xi): U\in\Omega_{\xi}, \ \ V=V(U), \ \ \xi< \lambda_1(U_L)\},
\nonumber
\end{equation}
where $\Omega_{\xi}$ is an open subset of $U$-space that depends on $\xi$ and $U_L$.
(The linearization of $W_0^u$ at a point in $S_0$ has a basis of eigenvectors, but
we can ignore them for now, noting only that the projection of $W_0^u$ onto $U$-space
contains a full neighborhood of $U_L$.
The function $V(U)$ is determined by solving the system \eqref{4.1.2}.)
Similarly,
\begin{equation}
S_2(U_R)=\{(U,V,\xi): U=U_R, \ \ V=0, \ \ \lambda_2(U_R)<\xi\}
\nonumber
\end{equation}
is a $1$-dimensional set, which has a $3$-dimensional stable manifold,
$W_{\varepsilon}^{s}(S_2(U_R))$, the perturbation of
\begin{equation}
W_{0}^{s}(S_2(U_R))=\{(U,V,\xi): U\in\Omega_{\xi}, \ \ V=V(U), \ \ \lambda_2(U_R)<\xi\}\,.
\nonumber
\end{equation}

Since every trajectory in $W_{\varepsilon}^u(S_0(U_L))\cap W_{\varepsilon}^s(S_2(U_R))$
tends to $U_R$ as $\tau \to \infty$ and to $U_L$ as $\tau \to -\infty$, our objective is to
show that these two $3$-dimensional manifolds intersect in the $5$-dimensional state space.

As an alternative for the same purpose, we focus attention on the shock layer, and specifically on the difficulties surrounding
the Rankine-Hugoniot relation, which normally is derived from equations \eqref{3.1.13}
and \eqref{3.1.15}, and replace $V$ in \eqref{4.1.2} by
$$ W= -V+F(U)-\xi U\,.$$
Also, from now on we treat $\varepsilon$ as a dynamical variable. Then we have the system
\begin{equation}
\begin{aligned}
\frac{dv}{d\tau}&=\frac{y}{v}-\xi v-w_1, \\
\frac{dy}{d\tau}&= \frac{1}{v}-\xi y-w_2, \\
\frac{dw_1}{d\tau}&=-\varepsilon v, \\
\frac{dw_2}{d\tau}&=-\varepsilon y, \\
\frac{d\xi}{d\tau}&=\varepsilon, \\
\frac{d\varepsilon}{d\tau}&=0.
\end{aligned}
\label{4.1.4}
\end{equation}

Each subspace $\varepsilon=$constant is invariant. Corresponding to the $3$-dimensional subsets $S_0$ and $S_2$ we have  now $4$-dimensional normally hyperbolic subsets which we write as
\begin{equation*}
\begin{aligned}
T_0&=\{(U,W,\xi, \varepsilon): \ \ \|U\|\leq\frac{1}{\delta}, \ \  W=F(U)-\xi U,
 \xi\leq \lambda_1(U)-\delta\},\\
T_2&=\{(U,W,\xi, \varepsilon): \ \ \|U\|\leq\frac{1}{\delta}, \ \ W=F(U)-\xi U,
 \lambda_2(U)+\delta\leq\xi\}\,.
\end{aligned}
\end{equation*}
The $1$-dimensional sets $S_0(U_L)$ and $S_2(U_R)$ are now
\begin{equation*}
\begin{aligned}
T_0^{\varepsilon}(U_L)&=\{(U,W,\xi,\varepsilon):  \ U=U_L, W=F(U_L)-\xi U_L,
 \xi\leq \lambda_1(U_L)-\delta,  \ \varepsilon  \ \text{fixed} \},\\
T_2^{\varepsilon}(U_R)&=\{(U,W,\xi,\varepsilon):  \ U=U_R, W=F(U_R)-\xi U_R,
 \xi\geq \lambda_2(U_R)+\delta, \ \varepsilon  \ \text{fixed} \},
\end{aligned}
\end{equation*}
and we rewrite the $3$-dimensional unstable manifold $W_\varepsilon^u(S_0(U_L))$ as
\begin{equation*}
W^{u}(T_0^{\varepsilon}(U_L))=\{(U,W, \xi, \varepsilon): \ \
U\in \Omega_{\xi}, \ \ W=W(U), \ \xi<\lambda_1(U_L), \ \ \varepsilon \ \ \text{fixed}\},
\end{equation*}
where now $W(U)$ denotes the solution of \eqref{4.1.4} corresponding to $U$.
Finally, the $3$-dimensional stable manifold $W^s_\varepsilon(S_2(U_R))$ becomes
a $3$-dimensional space
\begin{equation*}
W^{s}(T_2^{\varepsilon}(U_R))=\{(U,W,\xi, \varepsilon):
U\in\Omega_{\xi}, \ \ W=W(U), \ \lambda_2(U_R)<\xi, \ \ \varepsilon \ \ \text{fixed} \}\,.
\end{equation*}
As with the previous coordinates,
we look for a solution for fixed $\varepsilon>0$ that lies in the intersection of
$W^u(T_0^{\varepsilon}(U_L))$ and $W^s(T_2^{\varepsilon}(U_R))$.

Now we write down an expression for the inner solution,
motivated by the formal derivation given in Section \ref{formal}. The scaling \eqref{3.1.4} introduces a new variable
$Y=\left(
     \begin{array}{c}
       y_1 \\
       y_2 \\
     \end{array}
   \right)$
such that $$v=\frac{\varepsilon^2 y_2}{(y_1^{2/3}+\varepsilon^{\beta_3})^{3/2}}-\varepsilon^{\beta_4}, \ \ \ y=\frac{\varepsilon y_2^2}{(y_1^{16/15}+\varepsilon^{10})^{3/2}}-\varepsilon^{\beta_1}.$$
The system, with a time variable
$\eta = \tau/\varepsilon^3$  is now
\begin{equation}
\begin{aligned}
\frac{dy_1}{d\eta}&=\frac{5\varepsilon A(y_1, y_2, w_1, w_2, \xi, \varepsilon)}{2(4\varepsilon^{\beta_3}y_2 y_1^{6/15}-5\varepsilon^{10}y_2-y_2y_1^{16/15})}, \\
\frac{dy_2}{d\eta}&=\frac{\varepsilon B(y_1, y_2, w_1, w_2, \xi, \epsilon)}{(4\varepsilon^{\beta_3}y_2 y_1^{6/15}-5\varepsilon^{10}y_2-y_2y_1^{16/15})}, \\
\frac{dw_1}{d\eta}&=-\frac{\varepsilon^6 y_2}{(y_1^{2/3}+\varepsilon^{\beta_3})^{3/2}}+\varepsilon^{\beta_4+4},\\
\frac{dw_2}{d\eta}&=\varepsilon^{\beta_1+4}-\frac{\varepsilon^5y_2^2}{(y_1^{16/15}+\varepsilon^{10})^{3/2}}, \\
\frac{d\xi}{d\eta}&=\varepsilon^4, \\
\frac{d\varepsilon}{d\eta}&=0,
\end{aligned}
\label{4.1.5}
\end{equation}
where
\begin{align*}
A(y_1, y_2, \varepsilon, w_1, w_2, \xi)&=\frac{2(y_1^{2/3}+\varepsilon^{\beta_3})^4 y_1^{1/3}y_2^2}{\varepsilon(y_1^{16/15}+\varepsilon^{10})^{1/2}[y_2-\varepsilon^{\beta_4-2}(y_1^{2/3}+\varepsilon^{\beta_3})^{3/2}]}-\xi \varepsilon^2 y_1^{1/3}y_2(y_1^{2/3}+\varepsilon^{\beta_3})(y_1^{16/15}+\varepsilon^{10})\\
&-\frac{y_1^{1/3}(y_1^{2/3}+\varepsilon^{\beta_3})^{5/2}(y_1^{16/15}+\varepsilon^{10})^{5/2}}{\varepsilon y_2[y_2-\varepsilon^{\beta_4-2}(y_1^{2/3}+\varepsilon^{\beta_3})^{3/2}]}+2\xi y_1^{1/3}\varepsilon^{\beta_4}(y_1^{2/3}+\varepsilon^{\beta_3})^{5/2}(y_1^{16/15}+\varepsilon^{10})\\
&-2y_1^{1/3}w_1(y_1^{2/3}+\varepsilon^{\beta_3})^{5/2}(y_1^{16/15}+\varepsilon^{10})-\frac{2y_1^{1/3}\varepsilon^{\beta_1-1}(y_1^{2/3}+\varepsilon^{\beta_3})^4(y_1^{16/15}+\varepsilon^{10})}{\varepsilon[y_2-\varepsilon^{\beta_4-2}(y_1^{2/3}+\varepsilon^{\beta_3})^{3/2}]}\\
&+\frac{\varepsilon w_2 y_1^{1/3}(y_1^{2/3}+\varepsilon^{\beta_3})(y_1^{16/15}+\varepsilon^{10})^{5/2}}{y_2}-\frac{\varepsilon^{1+\beta_1} \xi y_1^{1/3}(y_1^{2/3}+\varepsilon^{\beta_3})(y_1^{16/15}+\varepsilon^{10})^{5/2}}{y_2},\\
B(y_1, y_2, \varepsilon, w_1, w_2, \xi)&=(y_1^{2/3}+\varepsilon^{\beta_3})^{3/2}(4\varepsilon^{\beta_3}y_2y_1^{6/15}+4y_2y_1^{16/15})\\
&\cdot\bigg[\frac{y_2^2(y_1^{2/3}+\varepsilon^{\beta_3})^{3/2}}{\varepsilon(y_1^{16/15}+\varepsilon^{10})^{3/2}[y_2-\varepsilon^{\beta_4-2}(y_1^{2/3}+\varepsilon^{\beta_3})^{3/2}]}-\frac{\varepsilon^2\xi [y_2-\varepsilon^{\beta_4-2}(y_1^{2/3}+\varepsilon^{\beta_3})^{3/2}]}{(y_1^{2/3}+\varepsilon^{\beta_3})^{3/2}}\\
&-\frac{\varepsilon^{\beta_1-1}(y_1^{2/3}+\varepsilon^{\beta_3})^{3/2}}{\varepsilon [y_2-\varepsilon^{\beta_4-2}(y_1^{2/3}+\varepsilon^{\beta_3})^{3/2}]}-w_1\bigg]\\
&-\frac{5}{2}\varepsilon(y_1^{16/15}+\varepsilon^{10})^{5/2}\bigg[\frac{(y_1^{2/3}+\varepsilon^{\beta_3})^{3/2}}{\varepsilon^2 [y_2-\varepsilon^{\beta_4-2}(y_1^{2/3}+\varepsilon^{\beta_3})^{3/2}]}-\frac{\xi \varepsilon y_2^2}{(y_1^{16/15}+\varepsilon^{10})^{3/2}}+\varepsilon^{\beta_1} \xi-w_2\bigg].
\end{align*}
The difficulty lies in matching the two outer solutions, expressed in $v$ and $y,$ satisfying the boundary conditions (\ref{3.1.3}), with an inner solution, expressed in $y_1,$ $y_2.$ System \ref{4.1.5} is of fundamental importance since it is not clear that one could use GSPT without some prior information about the asymptotics of the inner solution.

When $\varepsilon = 0$, the equation for $Y$ decouples from the rest of the
system, and is exactly \eqref{3.1.8}. Thus, \eqref{4.1.5} when $\varepsilon = 0$ is
\begin{equation}
\begin{aligned}
\frac{dy_1}{d\eta}&=\frac{5}{2}(\frac{y_1^{18/5}}{y_2^3}-2y_1^{7/5}), \\
\frac{dy_2}{d\eta}&=\frac{5}{2}\frac{y_1^{13/5}}{y_2^2}-4y_2y_1^{2/5}, \\
\frac{dw_1}{d\eta}&=-\frac{\varepsilon^6y_2}{(y_1^{2/3}+\varepsilon^{\beta_3})^{3/2}}\bigg|_{\varepsilon=0}=C(y_1,y_2),\\
\frac{dw_2}{d\eta}&=-\frac{\varepsilon^5 y_2^2}{(y_1^{16/15}+\varepsilon^{10})^{3/2}}\bigg|_{\varepsilon=0}=D(y_1, y_2), \\
\frac{d\xi}{d\eta}&=0, \\
\frac{d\varepsilon}{d\eta}&=0.
\end{aligned}
\label{4.1.6}
\end{equation}
The fact that $w_1$ and $w_2$ behave differently from each other is an
indication that the asymmetry in the generalized Rankine-Hugoniot relation will
enter into the analysis.

Desingularization of the system (by rescaling the time variable) on the set $y_1=0, y_2=0, \varepsilon=0$ shows that $E=\{(Y,W,\xi,\varepsilon): Y=0, \ \ \varepsilon=0\}$ is a $3$-dimensional space consisting
entirely of equilibria.
If we linearize at a point in $E$, we find that all $6$
eigenvalues are zero.
This is exactly the situation found by Schecter in \cite{Schecter} and a blow-up is necessary to resolve the behavior of the system near $E.$
\subsection{\textbf{The Blow-Up Construction}}\
\\
Under the change of variables
\begin{equation}
\begin{aligned}
y_1&=\bar{r}\bar{y}_1, \\
y_2&=\bar{r}^{11/15}\bar{y}_2, \\
w_1&=w_1, \\
w_2&=w_2, \\
\xi&=\xi, \\
\varepsilon&=\bar{r}\bar{\varepsilon}.
\end{aligned}
\label{4.2.1}
\end{equation}
with $|\bar{Y}|^2+\bar{\varepsilon}^2=1$,
the set $E$ becomes the set $\{\bar{r}=0\}$.
This set is now $5$-dimensional, in the $6$-dimensional $(\bar{Y},\bar{\varepsilon},\bar{r},W,
\xi)$-space
$\mathbf{X} = S^2\times \mathbb{R}_{+} \times \mathbb{R}^3$.
The system is also highly singular at $\{\bar{r}=0\}$,
but becomes non-singular upon division by $\bar{r}^{2/5}$.
Thus, we can study the dynamics of the transformed system on $\mathbf{X}$.
In terms of asymptotic structure, the change of variables \eqref{4.2.1} couples
the growth of $U$ to the limit $\varepsilon \to 0$ in the fashion predicted by the
formal asymptotics.
The range of $\bar{Y}$ and $\bar{\varepsilon}$ is confined to the unit sphere, but the dynamics of these variables can be explored since we can find invariant sets of low dimension of $\mathbf{X}$ and establish normally hyperbolicity. This will explain the connection between the bounded and unbounded parts of the singular shock. The homoclinic solution of Section \ref{formal} provides the inner dynamics and connecting the inner solution to the limit points $U_L$ and $U_R$
can now be pursued. 

We now define two intermediate points $q_L$ and $q_R$ which serve as bridge columns  connecting the inner and outer solutions. The connection between the homoclinic orbit, which can be identified as
the unique solution to \eqref{4.1.6} for which $w_{L2}-w_{R2} = k$ (the Rankine-Hugoniot
deficit, from equation \eqref{3.2.4}), and the states $U_L$ and $U_R$, which are limit
points of the manifolds $W^u(T_0^\varepsilon(U_L))$ and $W^s(T_2^\varepsilon(U_R))$
(for $\varepsilon \geq 0$), will be described. By making the transition from the unscaled variables $(U,W,\xi,\varepsilon)$
to the coordinate system in $\mathbf{X}$ we shall show that there is a unique orbit connecting $U_L$ with $q_L$. The connection between
$q_L$ and $q_R$ is via the homoclinic orbit and finally, $q_R$
connects to $U_R$ in the same manner as $U_L$ to $q_L$.

Because the beginning and ending connections are similar, in the sequel we will look only
at the first two steps.
Figure \ref{Figure4.2} gives a sketch of the key parts of the solution.

We begin with the definition of the intermediate points $q_L$ and $q_R$. In the coordinate system just introduced on $\mathbf{X}$, they are
\begin{align}
q_L&=(\frac{\bar{y}_2^{15/11}}{a_3},\bar{y}_2,0,0,W_L,s) \label{4.2.2}\\
q_R&=(\frac{\bar{y}_2^{15/11}}{a_2},\bar{y}_2,0,0,W_R,s) \label{4.2.3}
\end{align}
where we have written the coordinates in the order $(\bar{Y},\bar{\varepsilon},\bar{r},W,\xi)$;
$s$ is the speed of the singular shock, from \eqref{3.2.3};
$a_2$ and $a_3$ are the two roots
(in decreasing order) of
\begin{equation}\label{4.2.4}
a(a^{11/5}-2)=0 \;
\end{equation}
and $\bar{y}_2$ is the positive root of $\bar{y}_2^2+\frac{\bar{y}_2^{30/11}}{a_i^2}-1=0$
(so that $|\bar{Y}|^2+\bar{\varepsilon}^2=1$).
Finally,
\begin{equation}\label{4.2.5}
W_L=F(U_L)-sU_L, \quad W_R=F(U_R)-sU_R\,;
\end{equation}
we recall that $W=F(U_i)-\xi U_i$ ($i=L,R$) is the value of $W$ on the
invariant sets $T_0(U_L)$ and $T_2(U_R)$, so $q_L$ and $q_R$ are
specified by selecting the shock speed for $\xi$.
\subsection{\textbf{The First Stage of the Flow}}\
\\
From the description of the underlying planar system $U'=F(U)$ or $Y'=F(Y)$
and the sketch in Figure \ref{Figure3.1}, it is intuitively clear that the flow trajectories
are roughly parabolic.
Specifically, if we consider \eqref{4.1.4} with $\varepsilon =0$, $\xi =s$ and $W=W_L=
F(U_L)-sU_L$, then the equilibrium $U_L$ is a source.
\begin{proposition}
\label{proposition4.1}
The planar system $U'=F(U)-sU -W_L$ contains a negatively invariant region to the left of $U_L$, bounded by
\begin{align*}
\phi_1(v)&=y_L-E(v-v_L),\\
\phi_2(v)&=\frac{1}{s}\left(\frac{1}{v}-\frac{1}{v_L}\right)+y_L,
\end{align*}
where $E$ is such that
\begin{align*}
v_L \lambda_1(v_L,y_L)<E<v_L \lambda_2(v_L,y_L).
\end{align*}
\end{proposition}
\begin{proof}
A calculation of $U'$ along the curves $\phi_i$, similar to Lemma 3.2 in \cite{SSS},
gives the result.
\end{proof}
If we now consider \eqref{4.1.4} with $\varepsilon =0$, $\xi =s$ and $W=W_R=
F(U_R)-sU_R$, then the equilibrium $U_R$ is a sink.
\begin{proposition}
\label{proposition4.2}
The planar system $U'=F(U)-sU -W_R$ contains a positively invariant region to the right of $U_R$, bounded by
\begin{align*}
\phi_1(v)&=y_R-E(v-v_R),\\
\phi_2(v)&=sv(v-v_R)+\frac{y_R}{v_R}v,
\end{align*}
where $E$ is such that
\begin{align*}
v_R \lambda_1(v_R,y_R)<E<v_R \lambda_2(v_R,y_R),
\end{align*}
and a negatively invariant region to the left of $U_R$ bounded by
\begin{align*}
\phi_3(v)&=\frac{1}{s}\left(\frac{1}{v}-\frac{1}{v_R}\right)+y_R,
\end{align*}
and the coordinate axes.
\end{proposition}
\begin{proof}
A calculation of $U'$ along the curves $\phi_i$, similar to Lemma 3.2 in \cite{SSS},
gives the result.
\end{proof}
\begin{figure}
\psfrag{1}{$v$}
\psfrag{2}{$y$}
\psfrag{3}{$y^2=4v$}
\psfrag{4}{$U_L$}
\psfrag{5}{$y=\phi_1(v)$}
\psfrag{6}{$y=\phi_2(v)$}
\psfrag{7}{$y=\phi_2(v)$}
\psfrag{8}{$y=\phi_1(v)$}
\psfrag{9}{$U_R$}
\psfrag{10}{$y=\phi_3(v)$}
\begin{center}
\scalebox{0.9}{
\includegraphics{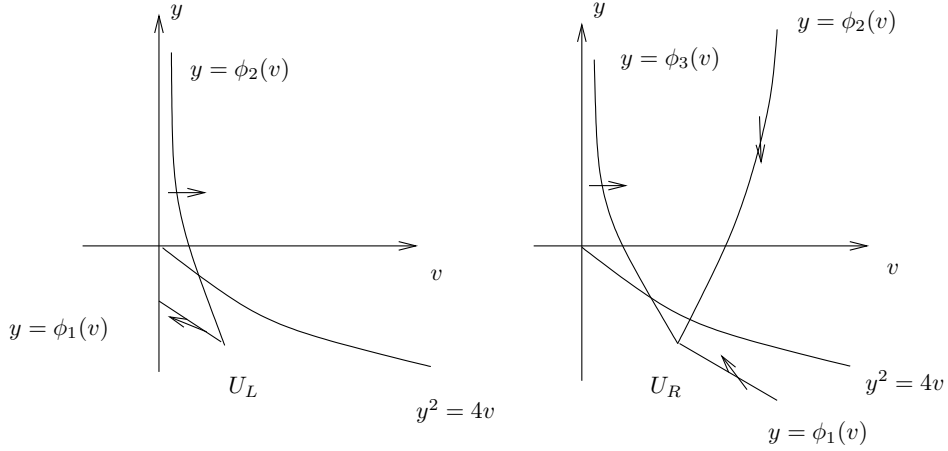}}
\caption{Invariant Regions.}
\label{Figure4.1}
\end{center}
\end{figure}
In particular, this means that trajectories within the curvilinear wedge between the
two curves of Proposition \ref{proposition4.1} and the $v$-axis all have $U_L$ as their $\alpha$-limits and similarly the trajectories within the open curvilinear wedge between the two curves $\phi_1$ and $\phi_2$ of Proposition \ref{proposition4.2} all have $U_R$ as their $\omega$-limits.
The trajectory beginning near $U_L$ becomes unbounded but the ratio $\frac{y_2^{15/11}}{y_1}$
remains bounded.
This motivates introducing a new coordinate chart on $\mathbf{X}$, which we will call
{\em Chart 2}, following Schecter's terminology in \cite{Schecter}.

\begin{figure}
\psfrag{e}{$\bar{\varepsilon}$}
\psfrag{y_2}{$y_2$}
\psfrag{u}{$u$}
\psfrag{xi}{$\xi$}
\psfrag{r}{$r$}
\psfrag{a}{$a$}
\psfrag{b}{$b$}
\psfrag{4}{$a_4$}
\psfrag{3}{$a_3$}
\psfrag{13}{\hspace*{-0.2cm}$\xi=s_{\text{singular}}$}
\psfrag{14}{$q_L$}
\psfrag{2}{$a_2, \ \xi=s_{\text{singular}}, \ \bar{q}_R$}
\psfrag{20}{$a_2, \ \xi=s_{\text{singular}}, \ {q}_R$}
\psfrag{1}{$a_1$}
\psfrag{S_0}{$S_0$}
\psfrag{S_2}{$S_2$}
\psfrag{K_1}{$\xi<\lambda_1(U)$}
\psfrag{K_2}{$\lambda_2(U)<\xi$}
\psfrag{U_L}{$U_L$}
\psfrag{U_R}{$U_R$}
\psfrag{5}{\hspace*{-2.8cm}$W^u(T_0^0(U_L)), \ \xi<s_{\text{singular}}$}
\psfrag{6}{\hspace*{-1.6cm}$W^u(N_0^0(U_L)), \  \xi<s_{\text{singular}}$}
\psfrag{7}{\hspace*{-0.5cm}$W^s(q_L)$}
\psfrag{8}{\hspace*{-0.2cm}$W^u(C_3)$}
\psfrag{9}{$W^s(C_2)$}
\psfrag{10}{\hspace*{-0.1cm}$W^u(q_R)$}
\psfrag{11}{\hspace*{1.4cm}$W^s(N_2^0(U_R)), \ \xi>s_{\text{singular}}$}
\psfrag{12}{\hspace*{-3.7cm}$W^s(T_2^0(U_R)), \ \xi>s_{\text{singular}}$}
\begin{center}
\scalebox{0.9}{
\includegraphics{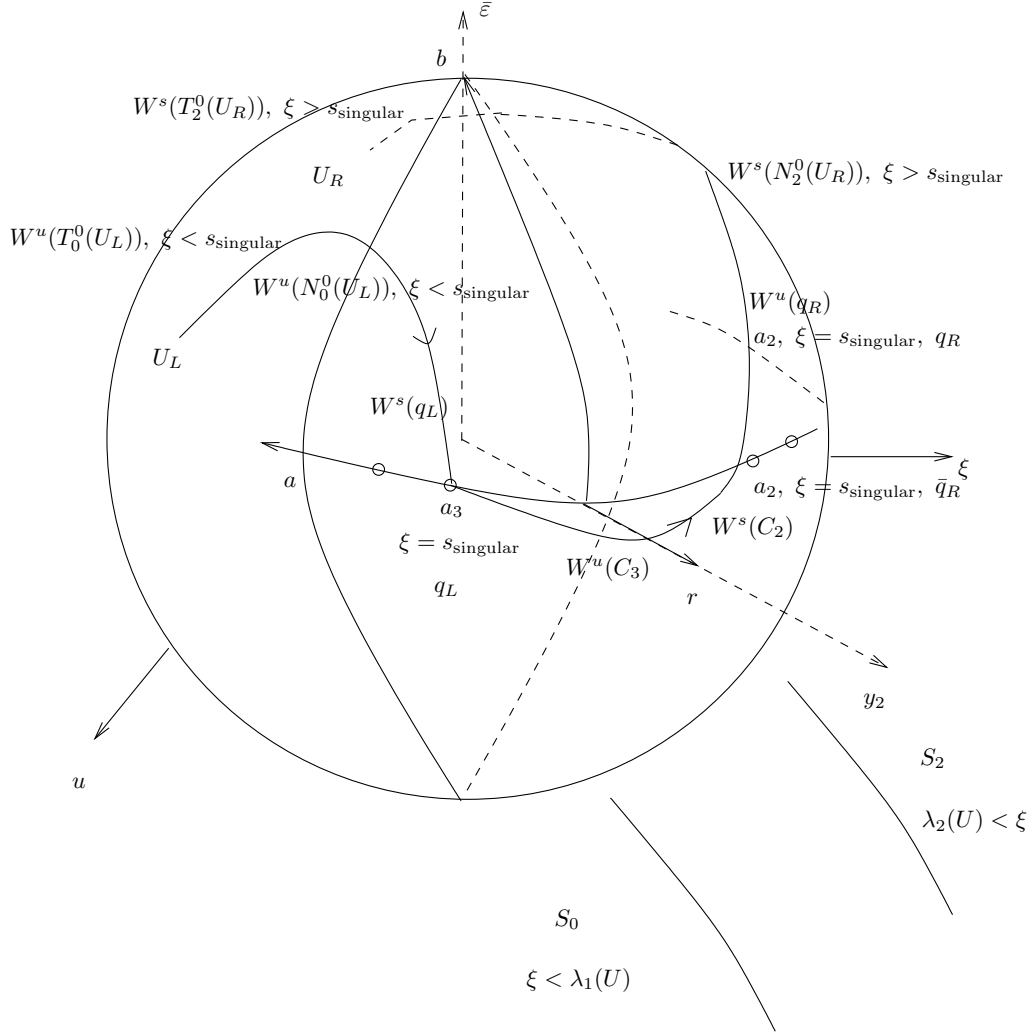}}
\caption{Chart 1 and 2.}
\label{Figure4.2}
\end{center}
\end{figure}
In terms of the coordinates $(\bar{Y},\bar{\varepsilon},\bar{r})$
(and, for reference, the scaled coordinates $(Y,\varepsilon)$ and the original
coordinates $(U,\varepsilon)$), we define, on the portion of
$\mathbf{X}$ where $\bar{y}_1$ $\bar{y}_2$ are positive,
\begin{equation}
\begin{aligned}
&a=\frac{\bar{y}_2^{15/11}}{\bar{y}_1}=\frac{y_2^{15/11}}{{y_1}}\left(\sim \frac{(y+\varepsilon^{\beta_1})^{10/11}}{(v+\varepsilon^{\beta_4})^{5/11}}\ \text{when} \ v, y \  \text{are large}\right), \\
&r=\bar{r}\bar{y}_2^{15/11}=y_2^{15/11}\left(\sim\frac{\varepsilon^{15/2}(y+\varepsilon^{\beta_1})^{75/22}}{(v+\varepsilon^{\beta_4})^{60/11}}\ \text{when} \ v, y \  \text{are large}\right),\\
&b=\frac{\bar{\varepsilon}}{{\bar{y}_2}^{15/11}}=\frac{\varepsilon}{{y_2}^{15/11}}\left(\sim\frac{(v+\varepsilon^{\beta_4})^{60/11}}{{\varepsilon^{13/2}(y+\varepsilon^{\beta_1})^{75/22}}}\ \text{when} \ v, y \  \text{are large}\right),
\end{aligned}
\label{4.3.1}
\end{equation}
and rescale the time variable to $\dfrac{r^{2/5}}{a^{39/15}}\eta$, which we will call $\zeta.$ This desingularizes the system (necessary to obtain a nontrivial flow) on the set $r=0, a=0$ but leaves it invariant. In these coordinates,
the system (\ref{4.1.5}) becomes
\begin{equation}
\begin{aligned}
\frac{d a}{d\zeta}&=\frac{a}{(4F-5G-1)}\bigg\{-\frac{75}{22}(1+G)^{5/2}\cdot\left(\frac{(1+F)^{3/2}}{(1-\Theta)}-\frac{\xi r^{39/15}a^{13/5}b^3}{(1+G)^{3/2}}-r^{26/15}ab^2w_2+r^{\beta_1}r^{26/15}ab^{2+\beta_1} \xi\right)\\
&+\frac{60}{11}a^{24/15}(1+F)^{5/2}\left(\frac{a^{3/5}(1+F)^{3/2}}{(1+G)^{3/2}(1-\Theta)}-\frac{r^{13/5} ab^3\xi(1-\Theta)}{(1+F)^{3/2}}-r^{13/15}bw_1-\frac{r^{2/15}r^{\beta_1-1}b^{\beta_1-1}}{a(1-\Theta)}(1+F)^{3/2}\right)\\
&-\frac{5a^{33/15}(1+F)^4}{(1+G)^{1/2}(1-\Theta)}+\frac{5}{2}\xi  r^{39/15}a^{39/15}b^3(1+F)(1+G)+\frac{5}{2}(1+F)^{5/2}(1+G)^{5/2}\\
&-5b^{\beta_4+1}r^{13/15}r^{\beta_4}a^{24/25}\xi(1+F)^{5/2}(1+G)\\
&+5bw_1 r^{13/15}a^{24/15}(1+F)^{5/2}(1+G)-\frac{5}{2}w_2r^{26/15}ab^2(1+F)(1+G)^{5/2}\\
&+\frac{5r^{\beta_1-1}r^{2/15}a^{3/5}b^{\beta_1-1}(1+F)^4(1+G)}{(1-\Theta)}+\frac{5}{2}\xi r^{\beta_1+1}r^{11/15}a b^{2+\beta_1}(1+F)(1+G)^{5/2}\bigg\}, 
\end{aligned}
\label{4.3.2}
\end{equation}
\begin{align*}
\frac{dr}{d\zeta}&=\frac{15r}{11(4F-5G-1)}\bigg\{-\frac{5}{2}(1+G)^{5/2}\\
&\cdot\left(\frac{(1+F)^{3/2}}{(1-\Theta)}-\frac{\xi r^{39/15}a^{13/5}b^3}{(1+G)^{3/2}}-r^{26/15}ab^2w_2+r^{\beta_1}r^{26/15}ab^{2+\beta_1} \xi\right)\\
&+4a^{24/15}(1+F)^{5/2}\left(\frac{a^{3/5}(1+F)^{3/2}}{(1+G)^{3/2}(1-\Theta)}-\frac{r^{13/5} ab^3\xi(1-\Theta)}{(1+F)^{3/2}}-r^{13/15}bw_1-\frac{r^{2/15}r^{\beta_1-1}b^{\beta_1-1}}{a(1-\Theta)}(1+F)^{3/2}\right)\bigg\}, \\
\frac{dw_1}{d\zeta}&=-\frac{r^{16/3}a^{54/15}b^6}{(1+F)^{3/2}}+a^{39/15}b^{4+\beta_4}r^{\beta_4}r^{54/15}, \\
\frac{dw_2}{d\zeta}&=a^{39/15}b^{4+\beta_1}r^{\beta_1}r^{54/15}-\frac{r^{67/15}a^{21/5}b^5}{(1+G)^{3/2}}, \\
\frac{d\xi}{d\zeta}&=r^{18/5}a^{39/15}b^4, \\
\frac{db}{d\zeta}&=\frac{15b}{11(4F-5G-1)}\bigg\{\frac{5}{2}(1+G)^{5/2}\\
&\cdot\left(\frac{(1+F)^{3/2}}{(1-\Theta)}-\frac{\xi r^{39/15}a^{13/5}b^3}{(1+G)^{3/2}}-r^{26/15}ab^2w_2+r^{\beta_1}r^{26/15}ab^{2+\beta_1} \xi\right)\\
&-4a^{24/15}(1+F)^{5/2}\left(\frac{a^{3/5}(1+F)^{3/2}}{(1+G)^{3/2}(1-\Theta)}-\frac{r^{13/5} ab^3\xi(1-\Theta)}{(1+F)^{3/2}}-r^{13/15}bw_1-\frac{r^{2/15}r^{\beta_1-1}b^{\beta_1-1}}{a(1-\Theta)}(1+F)^{3/2}\right)\bigg\},
\end{align*}
where $$F(a,r,b)=r^{\beta_3-1}r^{1/3}a^{2/3}b^{\beta_3}, \ \ G(a,r,b)=r^{134/15}a^{16/15}b^{10}, \ \ \Theta(a,r,b)=\frac{b^{\beta_4-2}r^{\beta_4-2}r^{4/15}}{a}(1+F)^{3/2}.$$
System (\ref{4.3.2}) plays a key role, since it contains all the dynamics of the
problem, scaled in a way that emphasizes the region where
the singular shock is formed.
In addition, this system also possesses an invariant manifold, which is
normally hyperbolic, and we are able to prove existence of a solution to the Dafermos
regularization, for small $\varepsilon$, by exhibiting a solution which is close to this invariant
manifold during part of its trajectory.

In the region of interest we require $r=0$ (which corresponds to $\varepsilon =0$) and $b=0$ to
find invariant manifolds, and then  we have an equilibrium of (\ref{4.3.2}) when $\frac{da}{d\zeta}=0$;
that is, when
$a$ is a root of the equation \eqref{4.2.4} introduced in the definition of $q_L$
and $q_R$.
The two roots of \eqref{4.2.4} are
\begin{equation*}
\begin{aligned}
a_2=2^{5/11}, \ \ a_3=0.
\end{aligned}
\end{equation*}
Using these roots, we define
$$P_j=\{(a,r,W,\xi,b): a=a_j, r=0, b=0\} \quad \textrm{for  }j=2, 3.$$
Each of these sets
is a $3$-dimensional manifold of equilibria, {\em corner equilibria} in Schecter's
definition \cite{Schecter}.
If we linearize (\ref{4.3.2}) at $a=a_j$, $r=b=0$, we find
a zero eigenvalue of multiplicity 3, with 3 linearly independent eigenvectors
lying in $P_j$.
There are three additional eigenvalues,
\begin{align*}
\lambda_2&=-\frac{16}{11}a_j^{11/15}+\frac{10}{11}\,,\\
\lambda_3&=-\frac{60}{11}a_j^{11/5}+\frac{75}{22}\,,\\
\lambda_4&=\frac{60}{11}a_j^{11/5}-\frac{75}{22}\,,
\end{align*}
and since
the corresponding eigenvectors, which are
\begin{align}
&R_2=(1, 0, 0, 0 , 0 , 0)\,,\nonumber\\
&R_3=(0, 1, 0, 0, 0, 0)\,,\label{4.3.3}\\
&R_4=(0, 0, 0, 0, 0, 1)\nonumber
\end{align}
respectively, are transversal to $P_j$, the
$P_j$ are normally hyperbolic manifolds.

We fix a point
$(a_3, 0, W_0, s_{\text{singular}}, 0)$  in $P_3$.
Then $\lambda_4<0<\lambda_2, \lambda_3$
so the point has a 1-dimensional stable manifold tangent to $R_4$.
Indeed, the stable manifold of any point with $r=b=0$ is
contained in the $2$-dimensional plane
$$\{(a,r,w_1,w_2,\xi,b): r=0, \ W=W_0, \ \xi=s_{\text{singular}}\}\,,$$
which is invariant under the flow \eqref{4.3.2}.
Thus the stable manifold of $P_3$ is tangent to
\begin{equation} \label{4.3.4}
\{(a,r,W,\xi,b): r=0, \ a=a_3\}
\end{equation}
at $P_3$.

Since $\lambda_2$ and $\lambda_3$ are positive at points of $P_3$,
each point has a $2$-dimensional unstable manifold tangent to the plane
spanned by $R_2$ and $R_3.$
(The same two eigenvalues, $\lambda_2$ and $\lambda_3$, are
negative on $P_2$.)
Thus $P_3$ has the $5$-dimensional unstable manifold
$$W^u(P_3)=\{(a,r,w_1,w_2,\xi,b): b=0\}.$$
The point $q_L$, identified earlier, is a particular point of $P_3$, with $W=W_L$
and $\xi = s_{\text{singular}}$.
(In Chart 2 coordinates, $q_L= (a_3,0,W_L,s_{\text{singular}},0)$ and $v=0,$ $y=0$ in the original variables.)
Through the $1$-dimensional stable manifold of $q_L\in P_3$, there is a unique
connection backwards in time to $U_L$, and through the $2$-dimensional unstable
manifold, $q_L$ connects forward to the singular orbit.
We state
\begin{proposition} \label{propq}
There is a unique orbit in the $2$-dimensional invariant plane
$$\{r=0,W=W_L,\xi=s_{\text{singular}}\}$$ that connects $q_L$ as $\zeta\to\infty$
with $U_L$ as $\zeta\to -\infty$.
Furthermore, in a neighborhood of $q_L$, we have $b>0$ along the orbit.
\end{proposition}
\begin{proof}
The proof is similar to the result of Schecter \cite{Schecter}, with details
motivated by
Theorem 3.1 of \cite{SSS}.
One can verify that, in one direction,
the stable manifold of $q_L$ is in the interior of the
negatively invariant region for $U_L$.
The inequality for $b$ follows from examining the eigenvector tangent to
the manifold at $q_L$. The manifolds are described in different coordinate
systems since the coordinate system of Chart 2 is not suitable for describing the entire
trajectory because $y$ (or $y_1,$ $y_2$ or $\bar{y}_1,$ $\bar{y}_2$) need not remain positive throughout
the trajectory.
\end{proof}

We now fix a point
$(a_2, 0, W_0, s_{\text{singular}}, 0)$  in $P_2$.
Then $\lambda_2, \lambda_3<0<\lambda_4$
so the point has a 1-dimensional unstable manifold tangent to $R_4$.
Indeed, the unstable manifold of any point with $r=b=0$ is
contained in the $2$-dimensional plane
$$\{(a,r,w_1,w_2,\xi,b): r=0, \ W=W_0, \ \xi=s_{\text{singular}}\}\,,$$
which is invariant under the flow \eqref{4.3.2}.
Thus the unstable manifold of $P_2$ is tangent to
\begin{equation} \label{4.3.4}
\{(a,r,W,\xi,b): r=0, \ a=a_2\}
\end{equation}
at $P_2$.

Since $\lambda_2$ and $\lambda_3$ are negative at points of $P_2$,
each point has a $2$-dimensional stable manifold tangent to the plane
spanned by $R_2$ and $R_3.$ Thus $P_2$ has the $5$-dimensional stable manifold
$$W^s(P_2)=\{(a,r,w_1,w_2,\xi,b): b=0\}.$$
The point $q_R$, identified earlier, is a particular point of $P_2$, with $W=W_R$
and $\xi = s_{\text{singular}}$. The point $q_R$ corresponds to $(v,y),$ $y^2=2v$ in the original variables. Through the $1$-dimensional unstable manifold of $q_R\in P_2$, there is a unique connection forward in time to $U_R.$ 

On the other hand, we need to show that through the $2$-dimensional stable manifold, $q_R$ connects backwards to the singular orbit. It should be noted that the connections between $U_L,$ $q_L,$ $q_R$ and $U_R$ 
do not solve the problem, since for example $q_L$ and $U_L$ are the $\omega$- and $\alpha$-limits
of a unique orbit, and thus are not themselves part of a longer connection between $U_L$
and $U_R$. To demonstration that connecting
orbits exist in the neighborhood of these invariant manifolds we use the
Corner Lemma to show that $U_L$ and $U_R$ can be connected when $\varepsilon>0$.

For this, we introduce an $1$-dimensional set that contains $q_L.$ We recall the definitions of $W_L$ and $W_R$, \eqref{4.2.5}, and of $q_L$ and $q_R$
in the Chart $2$ coordinate system
$$q_L=(a_3,0,W_L,s_\text{singular},0), \ \  q_R=(a_2,0,W_R,s_\text{singular},\infty).$$
In addition, we note that
using (\ref{3.2.7}) $w_{L1}=w_{R1},$ and  (\ref{3.2.8}), $w_{R2}=w_{L2}-k<w_{L2}$.

If we express
$q_L$ and $q_R$ in $Y,\varepsilon$ coordinates, they are points in  $E$
(the invariant set of equilibria of \eqref{4.1.6}).
Specifically,
$q_L=(0,W_L,s_{\text{singular}},0)$ and $q_R=(0,W_R,s_{\text{singular}},0)$.
Following the discussion of the homoclinic orbits in Section \ref{formal},
there is a unique solution of (\ref{4.1.6}) that connects the two points such that $k=\lim_{\varepsilon\rightarrow 0} \varepsilon^5 \int \frac{y_2^2}{(y_1^{16/15}+\varepsilon^{10})^{3/2}}\ d\eta.$
Write the solution as
$$(Y(\eta),W(\eta),s_{\text{singular}},0)\,,$$
with
$$w_1(\eta)=w_{L1}=w_{R1},\quad w_2(\eta)=w_{L2}-\lim_{\varepsilon\rightarrow 0} \varepsilon^5 \int_{-\infty}^{\eta} \frac{y_2(t)^2}{(y_1(t)^{16/15}+\varepsilon^{10})^{3/2}}\ dt=w_{L2}-k(\eta).$$
This can be
written in the coordinates of
 Chart 2, $(a, r ,W, \xi, b)$ as
\begin{equation}
q(\zeta)=(a(\zeta),r(\zeta),W(\zeta), s_{\text{singular}},b(\zeta))\,.
\label{4.4.3}
\end{equation}
Here $r(\pm\infty)=0$, $a(-\infty)=a_3$, $a(+\infty)=a_2$.
We note that
$q(-\infty)=q_L$, $q(+\infty)=q_R.$
Geometrically, $q(\zeta)$ lies in the $4$-dimensional subspace of $\mathbb{R}^6$
(in Chart 2 coordinates) with $w_1=w_{1L}=w_{2L}$ and $\xi=s.$ In addition there exists $q_M=(a_M, r_M, w_{1L},w_{2M}, s_{\text{singular}}, 0)$ which corresponds to $(v,y)=(0,\infty)$ in the original variables. 

We define
\begin{align*}
C_3&=\{(a,r,W,  \xi,b): a=a_3, \ r=0, W=F(U_L)-\xi U_L, \xi<\lambda_1(U_L), \ \ b=0\}\subseteq P_3,\\
D_3&=\{(a,r,W,  \xi,b): a=a_3, \ r=0, w_1=w_{L1}, \ w_2=w_{L2}-\lim_{\varepsilon\rightarrow 0} \varepsilon^5 \int_{-\infty}^{\zeta} \frac{y_2(t)^2}{(y_1(t)^{16/15}+\varepsilon^{10})^{3/2}}\ dt, \\ 
& \ \ \ \ \ \xi=s_{\text{singular}}, \ b=b(\zeta), \ \zeta\in\mathbb{R}\}, \\
E_3&=\{(a,r,W,  \xi,b): a=a_3, \ r=0, w_1=w_{L1}, \ w_2=w_{L2}-\lim_{\varepsilon\rightarrow 0} \varepsilon^5 \int_{-\infty}^{\zeta} \frac{y_2(t)^2}{(y_1(t)^{16/15}+\varepsilon^{10})^{3/2}}\ dt, \\ 
&\ \text{with} \ \zeta \ \text{such that} \ b=0, \ \xi=s_{\text{singular}}\}\subseteq D_3, \\
C_2&=\{(a,r,W,  \xi,b): a=a_2, \ r=0, W=F(U_R)-\xi U_R, \lambda_2(U_R)<\xi, \ \ b=\infty\},
\end{align*}
where we have not fixed the values of $W$ as we did to define $q_i$.
The stable manifold of $C_3$ is a $2$-dimensional surface in the $5$-dimensional space $r=0$;
it is the union of the stable manifolds of the points of $C_3.$ Since up to now we have not made use of the specific value of $\xi$ (beyond its relation to the eigenvalues of $dF(U_L)$), the results of Proposition \ref{propq} hold
at each point of $C_3$, and we have
(recalling that $T_0^0(U_L)$ is precisely the 1-dimensional set in which $\xi$ is
allowed to vary)
\begin{proposition}
In the coordinate system of Chart 2, the set
$W^u(T_0^0(U_L))$ takes the form
\begin{equation}\label{4.3.5}
W^u(N_0^0(U_L))=\{(a_\xi(\tau),0,W,\xi, b_\xi(\tau))\}\,,
\end{equation}
where $W=F(U_L)-\xi U_L$ for a fixed $\xi<\lambda_1(U_L)$ and $(a_\xi,b_\xi)$,
with $$a=\frac{y_2^{15/11}(\tau)}{{y_1(\tau)}} \ \text{and} \ b=\frac{\varepsilon(\tau)}{y_2^{15/11}(\tau)},$$ where $$v(\tau)=\frac{\varepsilon^2(\tau) y_2(\tau)}{(y_1^{2/3}(\tau)+\varepsilon^{\beta_3}(\tau))^{3/2}}-\varepsilon^{\beta_4}(\tau), \ \ \ y(\tau)=\frac{\varepsilon(\tau) y_2^2(\tau)}{(y_1^{16/15}(\tau)+\varepsilon^{10}(\tau))^{3/2}}-\varepsilon^{\beta_1}(\tau),$$ is the expression in Chart
2 coordinates of the solution of \eqref{4.3.2} with $\omega$-limit in $C_3$
for fixed $\xi$.
The intersection of $W^u(N^0_0(U_L))$ and $W^s(P_3)$ is an open
subset $Q_3$ of $W^s(C_3)$, namely the points of $W^s(C_3)$ with $b>0$.
\end{proposition}
\begin{proof}
The conclusion of Proposition \ref{propq}, which holds at each point of $C_3$, implies this
result.
The positivity of $b$ follows from the explicit scaling.
\end{proof}
The analogous result for $C_2$, and corresponding space $$W^s(N_2^0(U_R))=\{((a, r, W, \xi, b): (a,b)\in V_{\xi}, \ r=0, \ W=F(U_R)-\xi U_R, \ \lambda_2(U_R)<\xi \} ,$$
 are used to construct and analyze the
second half of the orbit.
For this purpose,
we note that $P_2$ has a $5$-dimensional stable manifold
$$W^s(P_2)=\{(a,r,W,\xi,b): b=0\}.$$
\subsection{\textbf{The Inner Solution}}
\label{section4.4}\
\\
We now seek the connection between $q_L$ and $q_R$. The curves $D_3$ of equilibria, in $P_3$ and $C_2$ are useful. The overcompression condition $s<\lambda_1(U_L)$ and $s>\lambda_2(U_R)$ is needed in this part or else the construction fails, because then $q_L$ is an endpoint of $C_3$ and we cannot verify Proposition \ref{prop6}, which we will need to apply the Corner Lemma at $q_L$ to match the inner with the outer solution.

The unstable manifold of $D_3$ has dimension three, and we have a description of its tangent space. It is spanned by the eigenvectors $R_2$ and $R_3$ of \eqref{4.3.3},
and can be written
\begin{equation}
\begin{aligned}
W^u(D_3)=\{(a,r,W,& \xi,b): w_1=w_{L1}, \ w_2=w_{L2}-\lim_{\varepsilon\rightarrow 0} \varepsilon^5 \int_{-\infty}^{\zeta} \frac{y_2(t)^2}{(y_1(t)^{16/15}+\varepsilon^{10})^{3/2}}\ dt, \\ 
& \ \ \ \ \ \xi=s_{\text{singular}}, \ b=b(\zeta), \ \zeta\in\mathbb{R}\}.
\end{aligned}
\label{4.4.1}
\end{equation}
\begin{figure}
\psfrag{1}{$y^2=4v$}
\psfrag{2}{$v$}
\psfrag{3}{$y$}
\psfrag{4}{$y^2=2v$}
\psfrag{5}{$U_L$}
\psfrag{6}{$U_R$}
\psfrag{7}{$a=a_3, r=b=0$}
\psfrag{8}{chart 1, outer solution}
\psfrag{9}{$a=a_2, r=b=0$}
\psfrag{10}{chart 2, inner solution}
\psfrag{11}{chart 2, inner solution}
\psfrag{12}{$a=a_2, r=0, b\neq 0$}
\psfrag{13}{chart 2, inner solution}
\psfrag{14}{chart 1, outer solution}
\begin{center}
\scalebox{0.8}{
\includegraphics{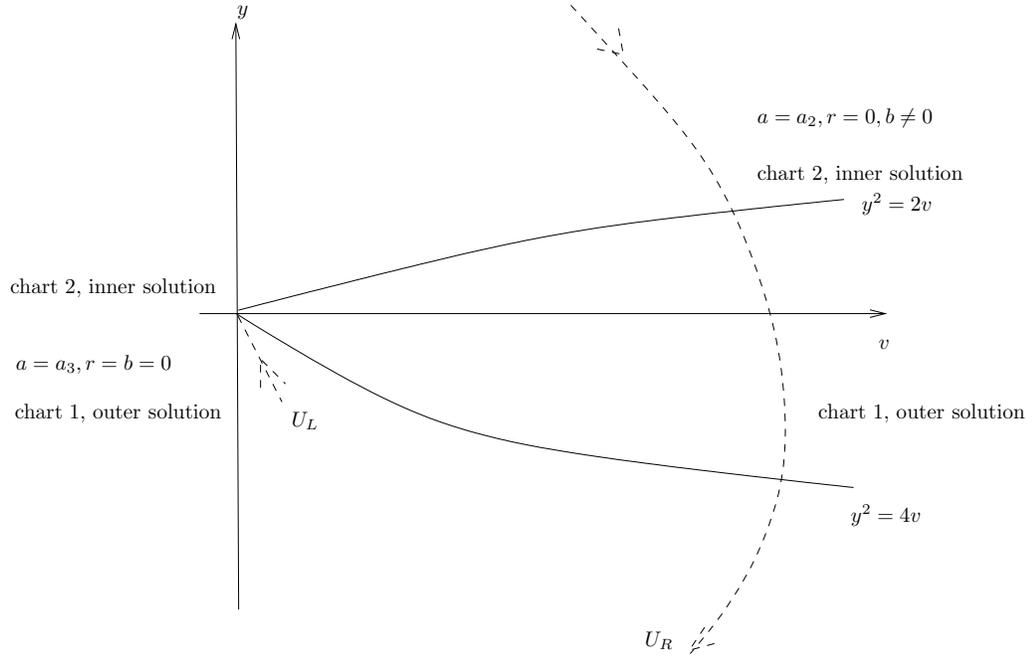}}
\caption{Solution when $\varepsilon=0$ in the $vy$-plane.}
\label{Figure4.3}
\end{center}
\end{figure}
\begin{figure}
\psfrag{1}{$a$}
\psfrag{2}{$r$}
\psfrag{3}{$b$}
\psfrag{4}{$a=a_3=b=r=0$}
\psfrag{5}{$a=a_2, r=b=0$}
\psfrag{6}{$a=a_2, r=0, b\neq 0$}
\psfrag{7}{$(\ref{2.1.4})$ has singular shocks}
\psfrag{8}{$(\ref{2.1.1})$ has singular shocks}
\begin{center}
\scalebox{0.8}{
\includegraphics{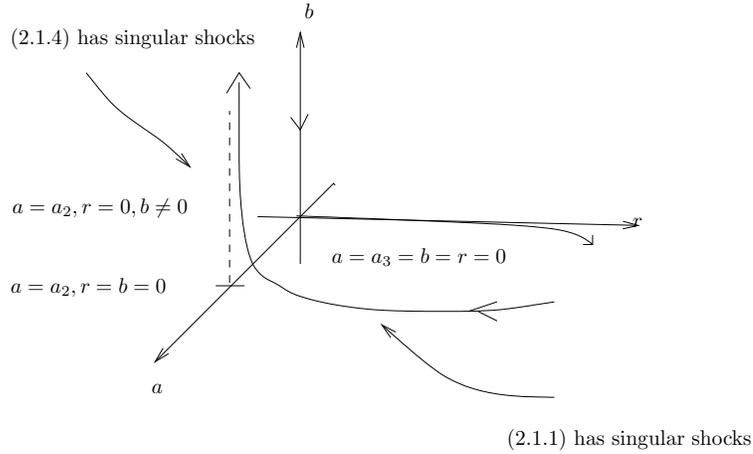}}
\caption{Solution when $\varepsilon=0$ in $arb$-space.}
\label{Figure4.4}
\end{center}
\end{figure}
\Remark
\label{remark1}
We observe that as $\zeta\rightarrow \infty,$ $r\rightarrow 0,$ $a\rightarrow a_2$ then $W^u(D_3)$ is tangent to 
\begin{equation} 
\{(a,r,W,\xi,b): a=a_2, \ r=0, \ W=W_R, \ \xi=s_\text{singular} \}
\end{equation} 
therefore $W^u(D_3)\cap W^s(N_2^0(U_R))\neq \emptyset.$ 
\Remark
\label{remark2}
$W^u(D_3) \supseteq W^u(C_3)\cap W^u(E_3)\neq \emptyset .$ 
\subsection{\textbf{Completion of the Result}}\
\\
The ingredients to be combined so as to synthesize the solution of the problem are now prepared. Three particular orbits have been constructed, each corresponding to the limit
$\varepsilon =0$: $A_1$, joining $U_L$ to $q_L$, $A_2$ joining $q_L$ to $q_R$,
and $A_3$ joining $q_R$ to $U_R$.
To show that a solution exists for $\varepsilon >0$, that will actually connect $U_L$
and $U_R$ via a solution of the equation, we need to
show that there is a solution, with $\varepsilon >0$, that is close to
the union of these three orbits.
The technique is to show that a solution close to $A_1$, in
$W^u(T_0^{\varepsilon}(U_L))$, will enter
$W^u(C_3)$, and similarly to match $W^u(D_3)$ with $W^s(T_2^{\varepsilon}(U_R))$.
We do this by verifying the conditions of the Corner Lemma
(Theorem 5.1 of Schecter \cite{Schecter}).
\begin{proposition} \label{prop6}
In the coordinate system of Chart 2, the sets
$W^u(T_0^{\varepsilon}(U_L))$ and $W^s(T_2^{\varepsilon}(U_R))$ will be denoted by $W^u(N_0^{\varepsilon}(U_L))$ and $W^s(N_2^{\varepsilon}(U_R)),$ respectively.

The $4$-dimensional set $W^u(N_0(U_L))=\cup_{0\leq \varepsilon \leq \varepsilon_0}W^u(N_0^{\varepsilon}(U_L))$ is transverse to $W^s(P_3)$ along $Q_3.$
\end{proposition}
\begin{proof}
When we calculate $W^u(N_0(U_L))$ at  $Q_3$ in the coordinate
system of Chart 2, we find
that
the tangent space to $W^u(N_0(U_L))$ is spanned by
\begin{align*}&(1,0,0,0,0,0), \\ &(0,0,0,0,0,1),\\
&(0,0,-v_{L},-y_{L},1,0), \\ &(0,1,0,0,0,0).
\end{align*}
The tangent space to $W^s(P_3)$ at the same point is spanned by
\begin{align*}&(0,0,1,0,0,0), \\ &(0,0,0,1,0,0).
\end{align*}
These six vectors are linearly independent; therefore transversality follows.
\end{proof}

\begin{figure}
\psfrag{b}{$b$}
\psfrag{r}{$r$}
\psfrag{a}{$a$}
\psfrag{p}{$p$}
\psfrag{q}{$q$}
\psfrag{q_L}{$q_L$}
\psfrag{Q_3}{$Q_3$}
\psfrag{Tilde}{$\tilde{N}_{\delta}$}
\psfrag{N_delta}{$N_{\delta}$}
\psfrag{P_3}{$P_3$}
\psfrag{U}{\hspace*{-0.2cm}$U$}
\psfrag{C_3}{$C_3$}
\psfrag{N}{$N$}
\begin{center}
\scalebox{0.9}{
\includegraphics{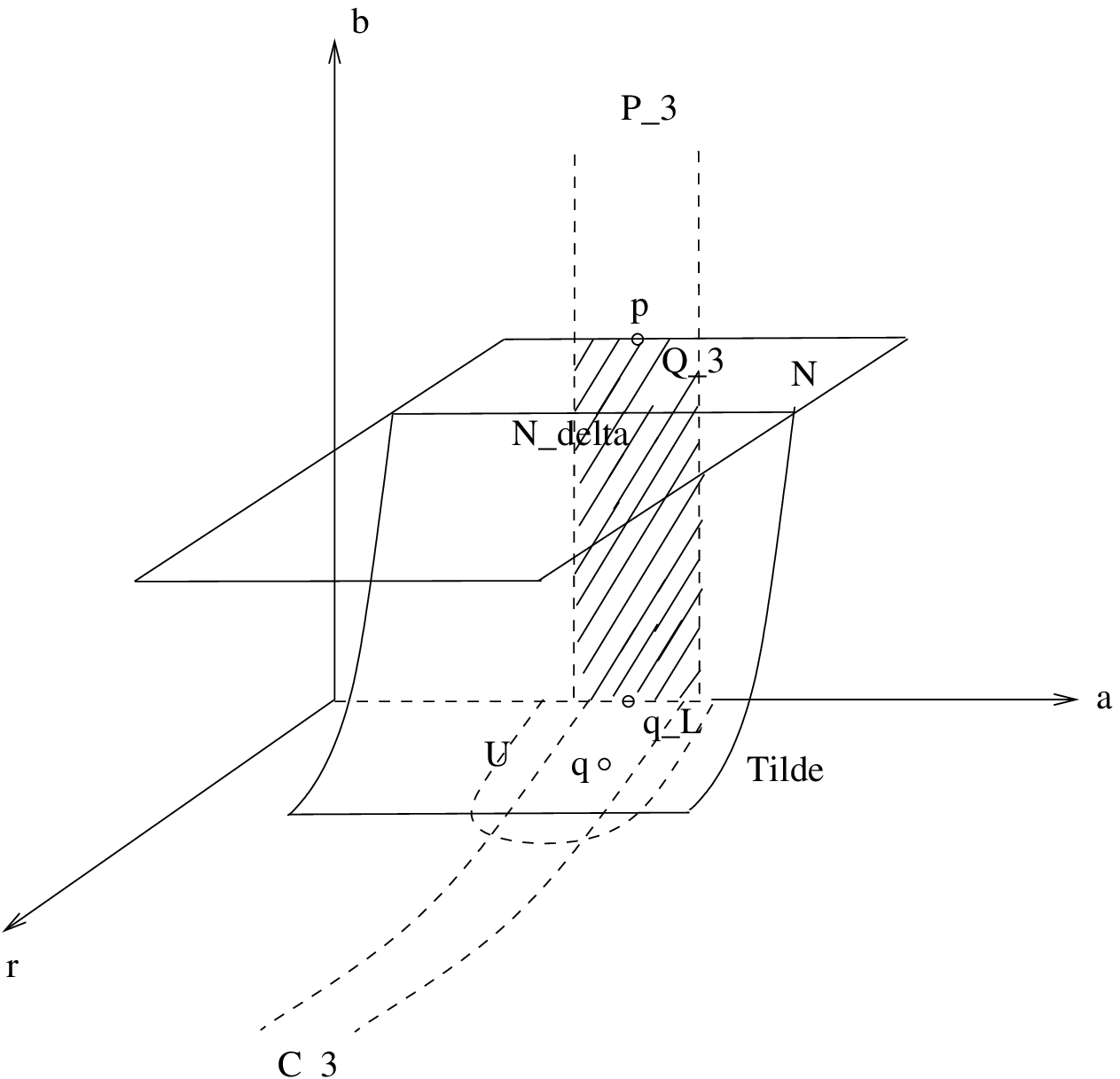}}
\caption{Corner Lemma.}
\label{Figure4.5}
\end{center}
\end{figure}

\begin{proof}[Proof of Theorem \ref{thm3.1}]
Proposition \ref{prop6} establishes the hypotheses of the
Corner Lemma \cite{Schecter}.
As Schecter showed in \cite{Schecter}, we have the 1-dimensional space
\begin{equation*}
W^s(q_L)=\{(a,r,W,\xi,b): r=0, \ \  W=W_L,  \xi=s_{\text{singular}}, \ \ a=a_3\}\,.
\end{equation*}
See Figure \ref{Figure4.5}.
We let $p \in W^s(q_L)\backslash\{q_L\}$, let $N$ be a 3-dimensional slice of
 $W^u(N_0(U_L))$
transverse to the vector field and to $W^s(P_3)$ at the point $p$;
let $N_{\delta}=N\cap\{r=\delta\}$, a 2-dimensional manifold;
let $q$ be in $W^u(C_3)$ with positive $r$ coordinate, and
let $U$ be a small neighborhood of $q$.

Then under the flow, $N_{\delta}$ becomes a 3-dimensional manifold $\tilde{N}_{\delta}$
(like $W^u(N_0^{\delta}(U_L))$) that passes near $q$. By the Corner Lemma,
\begin{equation}
\text{as}  \ \delta \rightarrow 0, \ \tilde{N}_{\delta}\cap U
\rightarrow W^u(C_3)\cap U \ \text{in the} \ C^1 \ \text{topology}.
\nonumber
\end{equation}
With the Lemma and Remarks in Section \ref{section4.4} we make the final match for the solution since
$W^u(N_0^{\varepsilon}(U_L))$ passes $q_L$ and arrives near
$q(-T)$ for $T>0,$ where
$q(\cdot)$ is given by (\ref{4.4.3}).
We then have a solution connecting $U_L$ and $U_R$.
As $\varepsilon \to 0$, this solution is unbounded.
This completes the proof of Theorem \ref{thm3.1}.
\end{proof}

\section{Acknowledgments} 
Foremost, the author would like to express her deepest thanks to her postdoctoral advisor Prof. Barbara Keyfitz for suggesting the problem and the change of variables leading to the system of equations (\ref{2.1.4}), for many illuminating discussions, and for her support and encouragement during this project. This work was started during the author's post-doctoral studies at the Ohio State University, whose support she also gratefully acknowledges. 

\bibliographystyle{amsplain}

\end{document}